\documentclass{amsart}
\usepackage{amssymb}
\usepackage{amsmath}
\usepackage{amsthm}
\usepackage{mathrsfs}
\usepackage{verbatim}
\newtheorem{theorem}{Theorem}[section]

\newtheorem{lemma}[theorem]{Lemma}
\newtheorem{corollary}[theorem]{Corollary}
\newtheorem*{corollary*}{Corollary}
\theoremstyle{definition}
\newtheorem{definition}[theorem]{Definition}

\theoremstyle{remark}
\newtheorem{remark}[theorem]{Remark}
\numberwithin{equation}{section}
\title[Cauchy-Dirichlet problem for Complex Monge-Amp\`ere flows]{The Pluripotential Cauchy-Dirichlet problem for the Complex Monge-Amp\`ere flow with a general measure on the right-hand side}
\author{Bowoo Kang}
\address{Department of Mathematical Sciences, KAIST, 291 Daehak-ro, Yuseong-gu, Daejeon
34141, South Korea}
\email{bou704@kaist.ac.kr}
\begin{document}
\begin{abstract}
    We show that the pluripotential Cauchy-Dirichlet problem for the complex Monge-Amp\`ere flow is solvable for the right-hand side of the form $dt \wedge d\mu$ where $d\mu$ is dominated by a Monge-Amp\`ere measure of a bounded plurisubharmonic function. In particular, we remove the strict positivity assumption on $d\mu$. We use this result to prove the parabolic version of the bounded subsolution theorem due to Ko{\l}odziej in pluripotential theory.
\end{abstract}
\maketitle
\section{Introduction}
The weak solution theory of the complex Monge-Amp\`ere equation has been well developed, starting from the pioneering works by Bedford and Taylor in \cite{BT76, BT82}. In 2021, Guedj, Lu and Zeriahi \cite{GLZ21-2}  successfully extended this theory into the parabolic setting and proved the existence of a solution to the pluripotential Cauchy-Dirichlet problem for the complex Monge-Amp\`ere flow in $(0, T) \times \mathbb{B}$ and $C^{1,1}$ regularity of the solution (in $\mathbb
{B}$), where $0 < T < +\infty$, $\mathbb{B}$ is a unit ball in $\mathbb{C}^n$, measure on the right-hand side has a density function $g$ satisfying $\log g \in C^{1,1}(\overline{\mathbb
B})$ and the boundary data satisfies some conditions. This is the parabolic analogue of the famous result by Bedford and Taylor \cite{BT76}. The authors \cite{GLZ21-2} also proved the solvability for the Cauchy-Dirichlet problem on a pseudoconvex domain with the measure (on the right-hand side) having a strictly positive $L^p$ density function for some $p > 1$, which is analogous to the famous result by Ko\l odziej \cite{Ko98}. \\
\indent The interest on this parabolic PDE was first motivated by its relationship with the K\"ahler-Ricci flow. In 1985, Cao \cite{Cao85} reproved the Yau's solution to the Calabi conjecture by solving the complex Monge-Amp\`ere flow on a compact K\"ahler manifold and proving its long-time existence. Since then, various parabolic Monge-Amp\`ere equations have been extensively studied, including Monge-Amp\`ere flows, J-flows, inverse Monge-Amp\`ere flows and so on. Recently, Phong and T\^o \cite{PT21} introduced the notion of parabolic C-subsolution, originated from the work by Guan \cite{Gu14} and Sz\'ekelyhidi \cite{Sz18}, to give the solvability condition for large class of parabolic Monge-Amp\`ere equations on compact Hermitian manifolds. Moreover, Picard and Zhang \cite{PZ20} proved the solvability and long-time existence of more general parabolic Monge-Amp\`ere equations in compact K\"ahler manifolds without the concavity condition on the operator. Smith \cite{SM20} extended this result to the case of compact Hermitian manifolds.\\
\indent Let $\Omega \Subset \mathbb{C}^n$ be a bounded strictly pseudoconvex domain and $T > 0$. Let $\partial_0\Omega_T := ([0, T) \times \partial \Omega) \cup (\{0\} \times \Omega)$ denote the parabolic boundary of $\Omega_T := (0, T) \times \Omega$. The complex Monge-Amp\`ere flow on $\Omega_T$ is 
\begin{align}\label{MAF}\det\left(\frac{\partial^2u}{\partial z_j \partial \overline{z}_k}\right) = e^{\frac{\partial u}{\partial t} + F\left(t, z, u\right)}g,
\end{align}
where $(0, T) \ni t \mapsto u(t, z)$ is a smooth family of strictly plurisubharmonic functions in $\Omega$, $F \in C^{\infty}([0, T) \times \Omega \times \mathbb{R})$ is increasing in the third variable, and $g \in C^{\infty}(\Omega)$ is strictly positive.\\
\indent The Cauchy-Dirichlet problem for the complex Monge-Amp\`ere flow is finding $u$ such that
\begin{align}\label{KR}
    \begin{cases}
        u \text{ satisfies } (\ref{MAF}) \text{ in } \Omega_T, \\
        \lim_{(t, z) \rightarrow (\tau, \zeta)}u(t, z) = h(\tau, \zeta) \text{ for all } (\tau, \zeta) \in \partial_0\Omega_T,
    \end{cases}
\end{align}
where $h$ is a Cauchy-Dirichlet boundary data defined on $\partial_0\Omega_T$. Hou and Li \cite{HL10} proved that if $g \equiv 1$, $h$ is smooth and some compatibility conditions are satisfied, then (\ref{KR}) admits a unique smooth solution. Later, Do \cite{Do17a} proved that if $g \equiv 1$, $h$ is smooth on $[0, T) \times \partial \Omega$ and $h(0, \cdot)$ is bounded and plurisubharmonic on $\overline{\Omega}$, then (\ref{KR}) admits a unique smooth solution. He also proved the existence and properties of the weak solution of (\ref{KR}) when $h(0, \cdot)$ has a zero or positive Lelong number \cite{Do16, Do17b}. The viscosity solution for (\ref{KR}) was studied by Eyssidieux, Guedj and Zeriahi \cite{EGZ15} (and also in the manifold setting \cite{EGZ16}). They proved that if $g$ is continuous, $(h(0, \cdot), g)$ is admissible in some sense, and $h$ does not depend on $t$, then (\ref{KR}) admits a unique viscosity solution. Recently, Do, Le and T\^o \cite{DLT20} extended this result and proved that when $g$ and $h$ depends on time and some conditions are satisfied, (\ref{KR}) admits a unique viscosity solution. In this case, $g$ should satisfy $\{z \in \Omega \mid g(t, z) = 0\} \subset \{z \in \Omega \mid g(s, z) = 0\}$ for all $0 < s < t < T$. Later, Do and Pham \cite{DP22} proved the existence and uniqueness of the viscosity solution to (\ref{KR}) when the sets $\{z \in \Omega \mid g(t, z) = 0\}$ may be disjoint, under some additional conditions. In this paper, we will only consider the case when $g$ does not depend on $t$. \\
\indent We focus on the results of \cite{GLZ21-2} in $(0, T) \times \Omega$. From now, we fix $0 < T < +\infty$. Even though the authors \cite{GLZ21-2} extended their results to the case when $T = \infty$, we will only consider the finite maximal time case. They first proved the estimates of the first and second derivatives of the envelope of subsolutions with respect to the time variable. Using the approximation method with the time regularity of the envelope, they proved that the envelope is indeed the solution. \\
\indent We explain the basic setting of \cite{GLZ21-2} in more detail, which will be also the basic setting of this paper. Assume that $u : \Omega_T \rightarrow [-\infty, +\infty)$ is a given function. We say that $u$ is locally uniformly Lipschitz in $(0, T)$ if for any subinterval $J \Subset (0, T)$, there exists a constant $\kappa_J > 0$ satisfying
\begin{align}\label{local uniform}
    u(t, z) \leq u(s, z) + \kappa_J \lvert t-s\rvert \text{ for all } t, s \in J \text{ and for all } z \in \Omega.
\end{align}
Similarly, the family $\mathcal{A}$ of functions mapping $\Omega_T$ to $[-\infty, +\infty)$ is said to be locally uniformly Lipschitz in $(0, T)$ if for any subinterval $J \Subset (0, T)$, there exists a constant $\kappa_J > 0$ satisfying (\ref{local uniform}) for all $u \in \mathcal{A}$. \\
\indent We say that $u$ is a \textit{parabolic potential} if
\begin{itemize}
    \item the map $u(t, \cdot) : z \mapsto u(t, z)$ is plurisubharmonic in $\Omega$;
    \item $u$ is locally uniformly Lipschitz in $(0, T)$.
\end{itemize}
We denote $\mathcal{P}(\Omega_T)$ the set of parabolic potentials on $\Omega_T$. If $u \in \mathcal{P}(\Omega_T) \cap L^{\infty}_{loc}(\Omega_T)$, then $dt \wedge (dd^cu)^n$ can be defined as a positive Radon measure on $\Omega_T$ \cite[Definition 2.2]{GLZ21-2}. \\
\indent We consider the following family of Monge-Amp\`ere flows
\begin{align*}
    \tag{CMAF} dt \wedge (dd^cu)^n = e^{\frac{\partial u}{\partial t}+F(t, z, u)}dt \wedge d\mu \text{ in } \Omega_T,
\end{align*}
where $\mu$ is a positive Borel measure on $\Omega$ satisfying $PSH(\Omega) \subset L^1_{loc}(\Omega, \mu)$, $F(t, z, r) : [0, T) \times \Omega \times \mathbb{R} \rightarrow \mathbb{R}$ is
\begin{itemize}
    \item continuous in $[0, T) \times \Omega \times \mathbb{R}$;
    \item increasing in $r$;
    \item locally uniformly Lipschitz in $(t, r)$;
    \item locally uniformly semi-convex in $(t, r)$,
\end{itemize}
and $u \in \mathcal{P}(\Omega_T) \cap L^{\infty}(\Omega_T)$ is the unknown function. As usual, $d = \partial + \overline{\partial}$ and $d^c = i(\overline{\partial} - \partial)$ so that $dd^c = 2i\partial \overline{\partial}$. Note that we understand (CMAF) in the pluripotential sense as defined in \cite{GLZ21-2}.\\
\indent A \textit{Cauchy-Dirichlet boundary data} is a function $h$ defined on $\partial_0\Omega_T$. We assume that $h$ is bounded and upper semi-continuous on $\partial_0\Omega_T$, and
\begin{itemize}
    \item $h\mid_{[0, T) \times \partial \Omega}$ is continuous;
    \item $h(\cdot, z)$ is locally uniformly Lipschitz in $(0, T)$ for each $z \in \partial\Omega$;
    \item $t\left\vert \frac{\partial h}{\partial t}(t, z)\right\vert \leq \kappa_h$ for all $(t, z) \in (0, T) \times \partial \Omega$ for some constant $\kappa_h > 0$;
    \item $\frac{\partial^2 h}{\partial t^2}(t, z) \leq C_ht^{-2}$ in the sense of distributions in $(0, T)$ for some $C_h > 0$;
    \item $h_0 := h(0, \cdot)$ is bounded, plurisubharmonic in $\Omega$, and satisfies
    \begin{align*}
        \lim_{\Omega \ni z \rightarrow \zeta}h(0, z) = h(0, \zeta) \text{ for all } \zeta \in \partial \Omega.
    \end{align*}
\end{itemize}
\indent \indent We fix some further notations and expressions which will be frequently used throughout this paper. First, we fix two such functions $F$ and $h$ introduced above. We set two constants related to $F$ and $h$ as follows:
\begin{itemize}
    \item $M_h := \sup_{\partial_0\Omega_T}\lvert h\rvert;$
    \item $M_F := \sup_{\Omega_T}F(\cdot, \cdot, M_h).$
\end{itemize}
Next, for a given sequence of Borel measures $\{\nu_j\}_{j = 1}^{\infty}$, we will say for convenience that $\nu_j \rightarrow \nu$ weakly as $j \rightarrow \infty$ for some Borel measure $\nu$, when the convergence is in the weak sense of Radon measures. We let $dV = 4^nn!\beta_n$ for $\beta_n = \left(\frac{i}{2}\right)^n\prod_{j = 1}^n dz_j \wedge d \overline{z}_j$ so that $(dd^c\varphi)^n = \det\left(\frac{\partial ^2 \varphi}{\partial z_j\partial \overline{z}_k}\right)dV$ if $\varphi \in C^2(\Omega)$. We let $d\ell$ denote the Lebesgue measure on $\mathbb{R}$.\\
\indent The Cauchy-Dirichlet problem for (CMAF) with the Cauchy-Dirichlet boundary data $h$ is finding $u \in \mathcal{P}(\Omega_T) \cap L^{\infty}(\Omega_T)$ such that
\begin{align}\label{CD}
    \begin{cases}
        &u \text{ satisfies (CMAF) in the pluripotential sense}, \\
        &\lim_{(t, z) \rightarrow (\tau, \zeta)}u(t, z) = h(\tau, \zeta) \text{ for all } (\tau, \zeta) \in [0, T) \times \partial \Omega, \\
        &\lim_{t \rightarrow 0^+}u(t, \cdot) = h_0 \text{ in } L^1(\Omega, d\mu).
    \end{cases}
\end{align}
Guedj, Lu and Zeriahi \cite{GLZ21-2} established a parabolic pluripotential theory and proved the existence and uniqueness of the solution of (\ref{CD}) when $d\mu = gdV$ for some $g \in L^p(\Omega)$ where $p > 1$ and $g > 0$ almost everywhere. They also proved that the pluripotential and viscosity solutions are equivalent under some conditions \cite{GLZ21-1}. We would like to mention that Guedj, Lu and Zeriahi \cite{GLZ20} also proved analogous results in the manifold setting, and Dang \cite{Da22} recently studied the pluripotential Monge-Amp\`ere flows in big cohomology class by extending the results in \cite{GLZ20}. \\
\indent Now, let us state our results. Our goal is to extend the results of \cite{GLZ21-2} for more general measures on the right-hand side, which may not be strictly positive or absolutely continuous with respect to the Lebesgue measure. The main result is as follows:
\begin{theorem}
    If there exists a function $\varphi \in PSH(\Omega) \cap L^{\infty}(\Omega)$ satisfying
    \begin{align*}
        \begin{cases}
            &(dd^c\varphi)^n \geq d\mu \text{ in } \Omega,\\
            &\lim_{z \rightarrow \partial \Omega}\varphi(z) = 0,
        \end{cases}
    \end{align*}
    then there exists a function $u \in \mathcal{P}(\Omega_T) \cap L^{\infty}(\Omega_T)$ satisfying
    \begin{align*}
        \begin{cases}
            &dt \wedge (dd^cu)^n = e^{\frac{\partial u}{\partial t}+F(t, z, u)}dt \wedge d\mu \text{ in } \Omega_T, \\
            &\lim_{(t, z) \rightarrow (\tau, \zeta)}u(t, z) = h(\tau, \zeta) \text{ for all } (\tau, \zeta) \in [0, T) \times \partial \Omega,\\
            &\lim_{t \rightarrow 0^+}u(t, \cdot) = h_0 \text{ in } L^1(\Omega, \mu).
        \end{cases}
    \end{align*}
\end{theorem}
This theorem implies the following corollary, which is the parabolic version of the Ko{\l}odziej's bounded subsolution theorem in \cite{Ko95}.
\begin{corollary}
    Let $h_1 \in C((0, T) \times \partial \Omega)$. Assume that there exists a function $v \in \mathcal{P}(\Omega_T) \cap L^{\infty}(\Omega_T)$ satisfying
    \begin{align*}
        \begin{cases}
            &dt \wedge (dd^cv)^n \geq e^{\frac{\partial v}{\partial t}+F(t, z, v)}dt \wedge d\mu \text{ in } \Omega_T, \\
            &\lim_{(t, z) \rightarrow (\tau, \zeta)}v(t, z) = h_1(\tau, \zeta) \text{ for all } (\tau, \zeta) \in (0, T) \times \partial \Omega.
        \end{cases}
    \end{align*}
    If $\mu$ is compactly supported or $h_1 \equiv 0$, then there exists a function $u \in \mathcal{P}(\Omega_T) \cap L^{\infty}(\Omega_T)$ satisfying
    \begin{align*}
        \begin{cases}
            &dt \wedge (dd^cu)^n = e^{\frac{\partial u}{\partial t}+F(t, z, u)}dt \wedge d\mu \text{ in } \Omega_T, \\
            &\lim_{(t, z) \rightarrow (\tau, \zeta)}u(t, z) = h(\tau, \zeta) \text{ for all } (\tau, \zeta) \in [0, T) \times \partial \Omega, \\
            &\lim_{t \rightarrow 0^+}u(t, \cdot) = h_0 \text{ in } L^1(\Omega, \mu).
        \end{cases}
    \end{align*}
\end{corollary}
\begin{remark}
    Note that the last boundary condition in (\ref{CD}) is weaker than
    \begin{align}\label{l1}
        \lim_{t \rightarrow 0^+}u(t, \cdot) = h_0 \text{ in } L^1(\Omega, dV),
    \end{align}
    which is the one in \cite{GLZ21-2}. In \cite{GLZ21-2}, the authors first proved the solvability under the Cauchy boundary condition in $L^1(\Omega, d\mu)$ and then improved it to (\ref{l1}) by using their assumption on the strict positivity of the density function. Since we are dealing with the more general right-hand side, we cannot improve the Cauchy boundary condition in (\ref{CD}) to (\ref{l1}) in general. Moreover, as we do not assume strict positivity of $\mu$, the parabolic Cauchy-Dirichlet problem (\ref{CD}) with the Cauchy boundary condition (\ref{l1}) is not solvable in general, which will be shown in Lemma $5.1$. Instead, we will show in Section $5.2$ that our Cauchy boundary condition in $L^1(\Omega, d\mu)$ can be improved to the one in \cite{GLZ21-2} when we assume some condition on the pair $(\mu, h_0)$.
\end{remark}
\begin{remark}
    We do not know yet about the uniqueness of the solution for (\ref{CD}). In \cite[Theorem 6.6]{GLZ21-2}, the authors proved 
    the comparison principle for complex Monge-Amp\`ere flows when the given measure has a strictly positive $L^p$ density function for some $p > 1$, which infers the uniqueness of the solution. Unfortunately, the proof of \cite[Theorem 6.6]{GLZ21-2} cannot be applied to our case, since we do not have the continuity of supersolutions and the absolute continuity of the given measure with respect to the Lebesgue measure.
\end{remark}
\indent 
\textbf{Organization}. In Section $2$, we recall some basic definitions and lemmas of the parabolic pluripotential theory in \cite{GLZ21-2}. In Section $3$, we prove some convergence lemmas related to the right-hand side of (CMAF), which contains the time derivatives of parabolic potentials. In Section $4$, we first prove our main result for compactly supported measures in Theorem $4.3$. By using the results and arguments in the proof of Theorem $4.3$, we get our main result in Theorem $4.5$. In Section $5$, we show that (\ref{CD}) is not solvable in general with the boundary condition (\ref{l1}). We then prove that we can get the boundary condition (\ref{l1}) of the solution under some conditions. \\
\mbox{}\\
\indent \textbf{Acknowledgements.} I would like to thank my advisor, Ngoc Cuong Nguyen, for suggesting this problem and providing invaluable guidance and support. I also would like to thank Do Hoang Son for his helpful comments on the issue of the Cauchy boundary condition.
\section{Preliminaries}
We briefly introduce basic definitions and lemmas related to the parabolic potentials and parabolic Monge-Amp\`ere operator. We refer the reader to \cite{GLZ21-2} for detailed proofs.
\subsection{Parabolic Monge-Amp\`ere Operator} We recall a parabolic version of the Chern-Levine-Nirenberg inequality (see e.g. \cite{Ko05}). 
\begin{lemma}[{\cite[Lemma 2.1]{GLZ21-2}}] \label{CLN} Fix $u \in \mathcal{P}(\Omega_T) \cap L^{\infty}_{loc}(\Omega_T)$ and $\chi$ a continuous test function on $\Omega_T$. Then the function
\begin{align*}
    \Gamma_{\chi} : t \mapsto \int_{\Omega}\chi(t, \cdot)(dd^cu(t, \cdot))^n
\end{align*}
is continuous in $(0, T)$. If $\text{supp}(\chi) \Subset E_1 \Subset E_2 \Subset \Omega_T$, then
\begin{align*}
    \sup_{0 \leq t < T}\left\vert \int_{\Omega}\chi(t, \cdot)(dd^cu(t, \cdot))^n\right\vert \leq C\sup_{\Omega_T}\lvert \chi\rvert (\sup_{E_2}\lvert u\rvert)^n,
\end{align*}
where $C > 0$ is a constant depending only on $(E_1, E_2, \Omega_T)$.
\end{lemma}
Note that not only the uniform inequality, but also the continuity of the map $\Gamma_\chi$ is a useful result which will be used later.
\begin{definition}[{\cite[Definition 2.2]{GLZ21-2}}]
    Fix $u \in \mathcal{P}(\Omega_T) \cap L^{\infty}_{loc}(\Omega_T)$. The map
    \begin{align*}
        C_c(\Omega_T) \ni \chi \mapsto \int_{\Omega_T}\chi dt \wedge (dd^cu)^n := \int_0^T dt \left(\int_{\Omega}\chi(t, \cdot)(dd^cu(t, \cdot))^n\right)
    \end{align*}
    defines a positive Radon measure in $\Omega_T$ denoted by $dt \wedge (dd^cu)^n$.
\end{definition}
\begin{lemma}
    Fix $u \in \mathcal{P}(\Omega_T) \cap L^{\infty}_{loc}(\Omega_T)$ and let $\{u_j\}_{j= 1}^{\infty}$ be a sequence of functions in $\mathcal{P}(\Omega_T) \cap L^{\infty}_{loc}(\Omega_T)$ such that for almost every $t \in (0, T)$,
    \begin{align*}
        (dd^cu_j(t, \cdot))^n \rightarrow (dd^cu(t, \cdot))^n \text{ weakly as } j \rightarrow \infty.
    \end{align*}
    Then
    \begin{align*}
        dt \wedge (dd^cu_j)^n \rightarrow dt \wedge (dd^cu)^n \text{ weakly as } j \rightarrow \infty.
    \end{align*}
\end{lemma}
\begin{proof}
    The conclusion can be shown by following the proof of \cite[Proposition 2.3]{GLZ21-2}.
\end{proof}
Let us recall the following definition of the capacity by Bedford and Taylor \cite[Definition 3.1]{BT82}. 
\begin{definition}
    For any Borel subset $E \subset \Omega$, 
    \begin{align*}
        cap(E, \Omega) := \sup\left\{\int_E (dd^cu)^n \mid u \in PSH(\Omega), -1 \leq u \leq 0\right\}.
    \end{align*}
\end{definition}
\begin{definition}[{\cite[Definition 4.23]{GZ17}}]
    A sequence of Borel functions $\{f_j\}_{j = 1}^{\infty}$ converges in capacity to a Borel function $f$ in $\Omega$ if for all $\delta > 0$ and all compact subsets $K \subset \Omega$,
    \begin{align*}
        \lim_{j \rightarrow \infty}cap(K \cap \{\lvert f_j-f\rvert \geq \delta\}, \Omega) = 0.
    \end{align*}
\end{definition}
By combining \cite[Theorem 1]{Xi96} (see also \cite[Theorem 4.26]{GZ17}) and Lemma $2.3$, we have the following lemma.
\begin{lemma} \label{continuity}
    Fix $u \in \mathcal{P}(\Omega_T) \cap L^{\infty}_{loc}(\Omega_T)$ and let $\{u_j\}_{j = 1}^{\infty}$ be a sequence of functions in $\mathcal{P}(\Omega_T) \cap L^{\infty}_{loc}(\Omega_T)$ such that for almost every $t \in (0, T)$,
    \begin{align*}
        u_j(t, \cdot) \rightarrow u(t, \cdot) \text{ in capacity as } j \rightarrow \infty.
    \end{align*}
    Then
    \begin{align*}
        dt \wedge (dd^cu_j)^n \rightarrow dt \wedge (dd^cu)^n \text{ weakly as } j \rightarrow \infty.
    \end{align*}
\end{lemma}
\subsection{Time Derivatives} Because of the local uniform Lipschitzness, the time derivatives of the parabolic potentials are well-defined. We first explain this well-definedness in more detail, which is a slight improvement of the one in \cite[Lemma 1.13]{GLZ21-2}.
\begin{lemma}
    Let $u \in \mathcal{P}(\Omega_T)$. Then there exists a Borel set $E \subset \Omega_T$ such that $\frac{\partial u}{\partial t}(t, z)$ exists for all $(t, z) \in \Omega_T \setminus E$ and
    \begin{align*}
        \int_{E\cap ((0, T) \times K)}dt \wedge (dd^c\varphi)^n = 0 
    \end{align*}
    for every compact set $K \subset \Omega$ and for every $\varphi \in PSH(\Omega)$ satisfying $-1 \leq \varphi \leq 0$.
\end{lemma}
\begin{proof}
    We first claim that $u^{-1}(\{-\infty\}) = (0, T) \times P$ for some pluripolar set $P \subset \Omega$. Indeed, assume that $u(t_0, z_0) = -\infty$ for some $(t_0, z_0) \in \Omega_T$. Let $P := (u(t_0, \cdot))^{-1}(\{-\infty\})$ be the polar set of the plurisubharmonic function $u(t_0, \cdot)$. We show that $u^{-1}(\{-\infty\}) = (0, T) \times P$. First, assume that $(t_1, z_1) \in u^{-1}(\{-\infty\})$. By the local uniform Lipschitzness of $u$, there exists $\kappa_1 > 0$ such that
    \begin{align*}
        u(t_0, z_1) \leq u(t_1, z_1) + \kappa_1\lvert t_0 - t_1\rvert,
    \end{align*}
    which implies that $u(t_0, z_1) = -\infty$. It infers that $(t_1, z_1) \in (0, T) \times P$. For the other direction, assume that $(t_2, z_2) \in (0, T) \times P$. Again by the local uniform Lipschitzness, there exists $\kappa_2 > 0$ such that
    \begin{align*}
        u(t_2, z_2) \leq u(t_0, z_2) + \kappa_2\lvert t_0-t_2\rvert,
    \end{align*}
    which implies that $u(t_2, z_2) = -\infty$. Therefore $(t_2, z_2) \in u^{-1}(\{-\infty\})$.
    \\
    \indent Next, we fix such a pluripolar set $P \subset \Omega$. By following the proof of \cite[Lemma 1.13]{GLZ21-2}, one can show that there exists a Borel set $E \subset \Omega_T$ such that $\frac{\partial u}{\partial t}(t, z)$ is well-defined for all $(t, z) \in \Omega_T \setminus E$ and
    \begin{align*}
        E_z := \{t \in (0, T) ~\mid~ (t, z) \in E\}
    \end{align*}
    satisfies $\ell(E_z) = 0$ for all $z \in \Omega \setminus P$. Therefore for every compact $K \subset \Omega$ and for every $\varphi \in PSH(\Omega)$ satisfying $-1 \leq \varphi \leq 0$,
    \begin{align*}
        \int_{E\cap ((0, T) \times K)}dt \wedge (dd^c\varphi)^n &= \int_{K\setminus P} (dd^c\varphi)^n \int_{E_z} d\ell + \int_{K \cap P}(dd^c\varphi)^n\int_{E_z}d\ell \\
        &= 0.
    \end{align*}
    The first equality holds by the Fubini's theorem, and the second equality holds by the fact that $cap(P, \Omega) = 0$. 
\end{proof}
We now introduce the local uniform semi-concavity, which was an important assumption on \cite[Proposition 2.9]{GLZ21-2} to obtain the convergence of the time derivatives.
\begin{definition}\mbox{}
    \begin{enumerate}
        \item [(1)] We say that $u : \Omega_T \rightarrow \mathbb{R}$ is locally uniformly semi-concave in $(0, T)$ if for any compact $J \Subset (0, T)$, there exists $C = C(J) > 0$ such that
    \begin{align}\label{concavity}
        t \mapsto u(t, z) - Ct^2 \text{ is concave in }J \text{ for all } z \in \Omega.
    \end{align}
        \item [(2)] We say that the family $\mathcal{A}$ of functions mapping from $\Omega_T$ to $\mathbb{R}$ is locally uniformly semi-concave in $(0, T)$ if for any compact $J \Subset (0, T)$, there exists $C = C(J) > 0$ such that (\ref{concavity}) holds for all $z \in \Omega$ and for all $u \in \mathcal{A}$.
    \end{enumerate}
\end{definition}
With the local uniform semi-concavity assumption, we can prove the following lemma for the convergence of the time derivatives, which is a slight improvement of \cite[Proposition 2.9]{GLZ21-2}.
\begin{lemma}
    Let $\{u_j\}_{j = 1}^{\infty}$ be a sequence of functions in $\mathcal{P}(\Omega_T)$ such that
    \begin{itemize}
        \item for a.e. $t \in (0, T)$, $u_j(t, \cdot) \rightarrow u(t, \cdot)$ in capacity as $j \rightarrow \infty$;
        \item is locally uniformly semi-concave in $(0, T)$.
    \end{itemize}
    Then there exists a Borel subset $E \subset \Omega_T$ such that for all compact subset $K \subset \Omega$ and for all $\varphi \in PSH(\Omega)$ satisfying $-1 \leq \varphi \leq 0$, $\lim_{j \rightarrow \infty}\frac{\partial u_j}{\partial t} = \frac{\partial u}{\partial t}$ in $((0, T) \times K) \setminus E$ and 
    \begin{align*}
        \int_{E \cap ((0, T) \times K)}dt \wedge (dd^c\varphi)^n = 0.
    \end{align*}
\end{lemma}
\begin{proof}
    Since $u_j(t, \cdot) \rightarrow u(t, \cdot)$ in capacity for a.e. $t \in (0, T)$ as $j \rightarrow \infty$, there exists a Borel subset $E_1 \subset \Omega_T$ such that for all compact subset $K \subset \Omega$ and for all $\varphi \in PSH(\Omega)$ satisfying $-1 \leq \varphi \leq 0$, $u_j \rightarrow u$ as $j \rightarrow \infty$ in $((0, T) \times K) \setminus E_1$ and
    \begin{align*}
        \int_{E_1 \cap ((0, T) \times K)}dt \wedge (dd^c\varphi)^n = 0.
    \end{align*}
     By following the proof of \cite[Proposition 2.9]{GLZ21-2}, one can show that there exists a Borel set $E_1 \subset E \subset \Omega_T$ satisfying all the conditions. Indeed, fix $K$ and $\varphi$. By Lemma $2.7$, there exists a Borel subset $E_2 \subset \Omega_T$ such that $\frac{\partial u_j}{\partial t}(t, z)$, $\frac{\partial u}{\partial t}(t, z)$ exist for all $j$ and $(t, z) \in \Omega_T \setminus E_2$, and
    \begin{align*}
        \int_{E_2 \cap ((0, T) \times K)}dt \wedge (dd^c\varphi)^n = 0,
    \end{align*}
    where $E_2$ is independent with the choice of $K$ and $\varphi$.
    Let $\tilde{E} := E_1 \cup E_2$ and $\tilde{E}_z := \{t \in (0, T) ~\mid~ (t, z) \in \tilde{E}\}$. By \cite[Lemma 2.8]{GLZ21-2}, for each $z \in \Omega$, there exists a Borel subset $E_z \subset (0, T)$ such that $\tilde{E}_z \subset E_z$, $\ell(E_z \setminus \tilde{E}_z) = 0$, and $\lim_{j \rightarrow \infty}\frac{\partial u_j}{\partial t}(t, z) = \frac{\partial u}{\partial t}(t, z)$ in $((0, T) \times K)\setminus (E_z \times \Omega)$. We define $E := E_z \times \Omega$. By the Fubini's theorem, we have
    \begin{align*}
        \int_{E \cap ((0, T) \times K)}dt \wedge (dd^c\varphi)^n &= \int_K(dd^c\varphi)^n \int_{E_z}d\ell\\ &= \int_{K}(dd^c\varphi)^n\int_{\tilde{E}_z}d\ell \\ &=\int_{\tilde{E} \cap (0, T) \times K)}dt \wedge (dd^c\varphi)^n = 0. 
    \end{align*}
    Since the above equality holds independently with the choice of $K$ and $\varphi$, we get the conclusion.
\end{proof}
\section{Convergence Lemmas}
Let $\{u_j\}_{j = 1}^{\infty}$ be a sequence of uniformly bounded parabolic potentials and let $\{\mu_j\}_{j = 1}^{\infty}$ be a sequence of positive Borel measures in $\Omega$. Assume that $d\mu_j \rightarrow d\mu$ weakly as $j \rightarrow \infty$.
We study about the sufficient condition for the weak convergence of
\begin{align}\label{goal}
    e^{\frac{\partial u_j}{\partial t}+F(t, z, u_j)}d\mu_j \rightarrow e^{\frac{\partial u}{\partial t}+F(t, z, u)}d\mu \text{ as } j \rightarrow \infty.
\end{align}
When $d\mu_j$'s and $d\mu$ all have strictly positive density functions which are in $L^p(\Omega)$ for some $p > 1$, $u_j \rightarrow u$ in $L^1(\Omega_T, \ell \otimes \mu)$ and $\{u_j\}_{j = 1}^{\infty}$ is locally uniformly semi-concave in $(0, T)$,  the convergence $(\ref{goal})$ is proved in \cite[Proposition 2.9]{GLZ21-2}. Our goal is to prove the similar result when $d\mu_j$'s and $d\mu$ are all Monge-Amp\`ere measures of some bounded plurisubharmonic functions.
\begin{lemma}\label{integration by parts}
    Let $\psi \in PSH(\Omega) \cap L^{\infty}(\Omega)$ and $u \in \mathcal{P}(\Omega_T) \cap L^{\infty}(\Omega_T)$. Then for every smooth test function $\chi$ on $\Omega_T$,
    \begin{align*}
        \int_{\Omega_T}\chi \frac{\partial u}{\partial t} dt \wedge (dd^c\psi)^n = -\int_{\Omega_T}u\frac{\partial \chi}{\partial t} dt \wedge (dd^c\psi)^n.
    \end{align*}
\end{lemma}
\begin{proof}
    Fix $\chi$ a smooth test function on $\Omega_T$. Let 
    \begin{align*}
        S(t) := \int_{\Omega}\chi(t, \cdot)u(t, \cdot)(dd^c\psi)^n.
    \end{align*}
    Assume that $\text{supp}(\chi) \subset J \times K \Subset \Omega_T$, where $J$ and $K$ are compact subsets of $(0, T)$ and $\Omega$ respectively. Since $u \in \mathcal{P}(\Omega_T)$, there exists a constant $\kappa_J(u) > 0$ such that
    \begin{align*}
        u(t, z) \leq u(s, z) + \kappa_J(u)\lvert t-s\rvert
    \end{align*}
    for all $t, s \in J$ and $z \in \Omega$. Thus for every $t_0, t_1 \in J$,
    \begin{align*}
        \lvert S(t_0) - S(t_1) \rvert &= \left\vert \int_K (\chi(t_0, \cdot)u(t_0, \cdot) - \chi(t_1, \cdot)u(t_1, \cdot))(dd^c\psi)^n\right\vert \\
        &\leq \int_K \lvert \chi(t_0, \cdot) - \chi(t_1, \cdot)\rvert \lvert u(t_0, \cdot)\rvert (dd^c\psi)^n \\
        &\quad+\int_K \lvert \chi(t_1, \cdot)\rvert\lvert u(t_0, \cdot) - u(t_1, \cdot) \rvert (dd^c\psi)^n \\
        &\leq C\lvert t-t_0\rvert
    \end{align*}
    where $C = \left(\int_K(dd^c\psi)^n\right)\left(\left\Vert \frac{\partial \chi}{\partial t}\right\Vert_{L^{\infty}(J \times K)}\lVert u\rVert_{L^{\infty}(\Omega_T)}+\lVert \chi\rVert_{L^{\infty}(\Omega_T)}\kappa_J(u)\right) > 0$,
    which implies that $S$ is Lipschitz. Therefore
    \begin{align*}
        \int_{\Omega_T}\chi \frac{\partial u}{\partial t} dt \wedge (dd^c\psi)^n &= \int_0^T dt \int_{\Omega}\chi(t, \cdot)\frac{\partial u}{\partial t}(t, \cdot)(dd^c\psi)^n \\
        &= \int_0^TS'(t)dt - \int_0^T dt \int_{\Omega}\frac{\partial \chi}{\partial t}(t, \cdot)u(t, \cdot)(dd^c\psi)^n \\
        &= -\int_0^T dt \int_{\Omega}\frac{\partial \chi}{\partial t}(t, \cdot)u(t, \cdot)(dd^c\psi)^n \\
        &= -\int_{\Omega_T}u\frac{\partial \chi}{\partial t} dt \wedge (dd^c\psi)^n.
    \end{align*}
    Here we used the integration by parts for the second identity. The third identity holds by the Lipschitzness of $S$.
\end{proof}
By Lemma \ref{integration by parts}, we get the following corollary which is similar to  \cite[Lemma 4.21]{GZ17}.
\begin{corollary}
    Assume that $\psi_j$, $\psi \in PSH(\Omega) \cap L^{\infty}(\Omega)$ satisfy
    \begin{align*}
        (dd^c\psi_j)^n \rightarrow (dd^c\psi)^n \text{ weakly as } j \rightarrow \infty.
    \end{align*}
    Then for $u \in \mathcal{P}(\Omega_T) \cap L^{\infty}(\Omega_T)$,
    \begin{align*}
        \frac{\partial u}{\partial t}dt \wedge (dd^c\psi_j)^n \rightarrow \frac{\partial u}{\partial t} dt \wedge (dd^c\psi)^n \text{ weakly as } j \rightarrow \infty.
    \end{align*}
\end{corollary}
\begin{proof}
    Fix $\chi$ a smooth test function on $\Omega_T$ and $t \in (0, T)$.
    Since $u(t, \cdot) \in PSH(\Omega) \cap L^{\infty}(\Omega)$, it follows from \cite[Lemma 4.21]{GZ17} that 
    \begin{align*}
    \int_{\Omega}\frac{\partial \chi}{\partial t}(t, \cdot)u(t, \cdot)(dd^c\psi_j)^n \rightarrow \int_{\Omega}\frac{\partial \chi}{\partial t}(t, \cdot)u(t, \cdot)(dd^c\psi)^n \text{ as } j \rightarrow \infty.
    \end{align*}
     By the elliptic Chern-Levine-Nirenberg inequality, $\left\{\int_{\Omega}\frac{\partial \chi}{\partial t}(t, \cdot)u(t, \cdot)(dd^c\psi_j)^n\right\}_{j = 1}^{\infty}$ is uniformly bounded. Therefore by the Lebesgue's convergence theorem for the integration with respect to $dt$ and Lemma $3.1$, we have
    \begin{align*}
        -\int_0^T dt \int_{\Omega}\frac{\partial \chi}{\partial t}(t, \cdot)u(t, \cdot)(dd^c\psi_j)^n \rightarrow &-\int_0^Tdt\int_{\Omega}\frac{\partial \chi}{\partial t}(t, \cdot)u(t, \cdot)(dd^c\psi)^n \\
        &= \int_{\Omega_T}\chi \frac{\partial u}{\partial t} dt \wedge (dd^c\psi)^n
    \end{align*}
    as $j \rightarrow \infty$. Therefore $\frac{\partial u}{\partial t}dt \wedge (dd^c\psi_j)^n \rightarrow \frac{\partial u}{\partial t} dt \wedge (dd^c\psi)^n$ as $j \rightarrow \infty$ in the weak sense of distributions in $\Omega_T$.
\end{proof}
We emphasize that the following lemma is the technical core of this paper. This improves the convergence result in \cite[Proposition 2.9]{GLZ21-2}. The proof is inspired by the linearization of the time derivatives used in \cite[Section 4]{GLZ21-2}. 
\begin{lemma}
    Assume that $\psi_j, \psi \in PSH(\Omega) \cap L^{\infty}(\Omega)$ satisfy
    \begin{itemize}
        \item $(dd^c\psi_j)^n \rightarrow (dd^c\psi)^n \text{ weakly as } j \rightarrow \infty$;
        \item $\{\psi_j\}_{j = 1}^{\infty}$ is uniformly bounded.
    \end{itemize}
    Assume that $u_j, u \in \mathcal{P}(\Omega_T) \cap L^{\infty}(\Omega_T)$ satisfy
    \begin{itemize}
        \item for almost every $t \in (0, T)$, $u_j(t, \cdot) \rightarrow u(t, \cdot)$ in capacity as $j \rightarrow \infty$;
        \item $\{u_j\}_{j = 1}^{\infty}$ is locally uniformly bounded;
        \item there exists a constant $\kappa_0 > 0$ satisfying
        \begin{align*}
            t\left\vert \frac{\partial u_j}{\partial t}(t, z)\right\vert \leq \kappa_0 \text{ for all } j
        \end{align*}
        for all $(t, z) \in \Omega_T$ where $\frac{\partial u_j}{\partial t}$ is well-defined;
        \item there exists a constant $C_0>0$ satisfying
        \begin{align*}
            \frac{\partial^2u_j}{\partial t^2}(t, z) \leq C_0t^{-2} \text{ for all } j
        \end{align*}
        in the sense of distributions in $(0, T)$ for all $z \in \Omega$.
    \end{itemize}
    Then
    \begin{align*}
        e^{\frac{\partial u_j}{\partial t}+F(t, z, u_j)}dt \wedge (dd^c\psi_j)^n \rightarrow e^{\frac{\partial u}{\partial t} + F(t, z, u)}dt \wedge (dd^c\psi)^n \text{ weakly }
    \end{align*}
    as $j \rightarrow \infty$.
\end{lemma}
\begin{proof}
    Fix $\chi$ a smooth positive test function in $\Omega_T$ and assume that $supp(\chi) \Subset J \times K \Subset \Omega_T$. By Lemma $2.7$, there exists a sequence $\{E_j\}_{j = 1}^{\infty}$ of Borel subsets of $\Omega_T$ such that for each $j$, $\frac{\partial u_j}{\partial t}(t, z)$ exists for all $(t, z) \in \Omega_T \setminus E_j$ and
    \begin{align*}
        \int_{E_j\cap ((0, T) \times K)} dt \wedge (dd^c\varphi)^n = 0
    \end{align*} for all $\varphi \in PSH(\Omega)$ satisfying $-1 \leq \varphi \leq 0$. Let $E = \bigcup_{j = 1}^{\infty}E_j$. Then $E \subset \Omega_T$ is a Borel subset such that for all $j$,
    \begin{align}\label{property of E}
        \int_{E \cap ((0, T) \times K)} dt \wedge (dd^c\psi_j)^n = \int_{E \cap ((0, T) \times K)}dt \wedge (dd^c\psi)^n = 0
    \end{align}
    and $\frac{\partial u_j}{\partial t}(t, z)$ exists for all $(t, z) \in \Omega_T \setminus E$. Hence we have 
    \begin{align}\label{core inequality 3}
        t\left\vert \frac{\partial u_j}{\partial t}(t, z)\right\vert \leq \kappa_0
    \end{align}
    for all $j$ and for all $(t, z) \in \Omega_T \setminus E$.\\
    \indent Note that $u_j(t, \cdot) \rightarrow u(t, \cdot)$ in capacity as $j \rightarrow \infty$ for a.e. $t \in (0, T)$ and the family $\{u_j\}_{j = 1}^{\infty}$ is locally uniformly semi-concave in $(0, T)$. Thus by Lemma $2.9$, there exist a Borel subset $G \subset \Omega_T$ such that for all $j$,
    \begin{align*}
        \int_{G \cap ((0, T) \times K)}dt \wedge (dd^c\psi_j)^n = \int_{G \cap ((0, T) \times K)}dt \wedge (dd^c\psi)^n = 0
    \end{align*} 
        and $\frac{\partial u_j}{\partial t}$ pointwisely converges to $\frac{\partial u}{\partial t}$ in $((0, T) \times K) \setminus G$ as $j \rightarrow \infty$.    
    Hence we have $t\left\vert \frac{\partial u}{\partial t}(t, z)\right\vert \leq \kappa_0$ for all $(t, z) \in ((0, T) \times K) \setminus G$. \\
    \indent Next, since $u_j \rightarrow u$ weakly as $j \rightarrow \infty$, $\frac{\partial^2 u}{\partial t^2}(t, z) \leq C_0t^{-2}$ holds in the sense of distributions in $(0, T)$ for all $z \in \Omega$. Indeed, let $\alpha$ and $\eta$ be smooth positive test functions on $(0, T)$ and $\Omega$ respectively. As $\frac{\partial^2 u_j}{\partial t^2}(t, z) \leq C_0t^{-2}$ holds in the sense of distributions in $(0, T)$ for all $z \in \Omega$ and $j$, we have
    \begin{align*}
         \int_{\Omega}\eta(z)dV\left(\int_{0}^T \frac{\partial^2\alpha}{\partial t^2}(t)u_j(t, \cdot)dt\right) \leq \int_{\Omega}\eta(z)dV\left(\int_0^T C_0\alpha(t)t^{-2}dt\right).
    \end{align*}
    It follows from the weak convergence of $u_j \rightarrow u$ that
    \begin{align*}
        \int_{\Omega}\eta(z) dV \left(\int_0^T\frac{\partial^2\alpha}{\partial t^2}(t) u(t, \cdot)dt\right) \leq \int_{\Omega}\eta(z)dV\left(\int_0^T C_0\alpha(t)t^{-2}dt\right).
    \end{align*}
    This implies that for $dV$-a.e. $z \in \Omega$,
    \begin{align}\label{semi concavity}
        \int_0^T \frac{\partial^2\alpha}{\partial t^2}(t)u(t, z)dt \leq \int_0^T C_0\alpha(t)t^{-2}dt.
    \end{align}
    As $U(z) := \int_0^T \frac{\partial ^2\alpha}{\partial t^2}(t)u(t, z)dt$ is a plurisubharmonic function in $\Omega$, (\ref{semi concavity}) holds for all $z \in \Omega$.
    \\
    \indent Fix $\varepsilon > 0$. We first show that
    \begin{align}\label{R1}
        \left(e^{\frac{\partial u_j}{\partial t}+F(t, z, u_j)} - e^{\frac{\partial u}{\partial t}+F(t, z, u)}\right)dt \wedge (dd^c\psi_j)^n \rightarrow 0 \text{ weakly }
    \end{align}
    as $j \rightarrow \infty$. \\
    \indent In fact, for sufficiently small $0 < \delta < T - \sup_{t \in J}t$ which will be chosen later, 
    \begin{align*}
        &\left\vert \int_{\Omega_T}\chi \left(e^{\frac{\partial u_j}{\partial t}+F(t, z, u_j)}-e^{\frac{\partial u}{\partial t}+F(t, z, u)}\right) dt \wedge (dd^c\psi_j)^n \right\vert \\
        &\quad \leq C_1\int_{\Omega_T\setminus (E \cap G)}\chi \left\vert \frac{\partial u_j}{\partial t} - \frac{\partial u}{\partial t}\right\vert dt \wedge (dd^c\psi_j)^n \\ &\quad \quad +C_1\int_{\Omega_T\setminus (E\cap G)}\chi \lvert F(t, z, u_j) - F(t, z, u)\rvert dt \wedge (dd^c\psi_j)^n\\
        &\quad \leq C_1\int_{(J \times K) \setminus E}\chi \left\vert \frac{\partial u_j}{\partial t}(t, z) - \frac{u_j(t+\delta, z) - u_j(t, z)}{\delta}\right\vert dt \wedge (dd^c\psi_j)^n \\
        &\quad \quad + C_1\int_{(J \times K)\setminus G}\chi\left\vert\frac{\partial u}{\partial t}(t, z) - \frac{u(t+\delta, z) - u(t, z)}{\delta} \right\vert dt \wedge (dd^c\psi_j)^n \\
        &\quad \quad + C_1\int_{(J \times K)
        }\chi \left\vert \frac{u_j(t+\delta, z) - u_j(t, z)}{\delta} - \frac{u(t+\delta, z) - u(t, z)}{\delta}\right\vert dt \wedge (dd^c\psi_j)^n \\
        &\quad \quad + C_1\int_{\Omega_T
        }\chi \lvert F(t, z, u_j) - F(t, z, u)\rvert dt \wedge (dd^c\psi_j)^n\\
        &\quad =: I_1 + I_2 + I_3 + I_4,
    \end{align*}
    where $C_1 = \sup_{t \in J}e^{\frac{\kappa_0}{t}+M_F} > 0$. We used the Lipschitzness of the exponential function for the first inequality.\\
    \indent For each $j$, let $v_j(t, z) := u_j(t, z) + C_0 \log t$. By the assumption on $\frac{\partial^2u_j}{\partial t^2}$, $v_j(t, z)$ is (weakly) concave in $(0, T)$ for all $z \in \Omega$. Hence, \begin{align}\label{core inequality}
        \delta\frac{\partial v_j}{\partial t}(t, z) \geq v_j(t+\delta, z) - v_j(t, z)
    \end{align} for all $(t, z) \in \Omega_T\setminus E$ since for all $j$, $\frac{\partial u_j}{\partial t}(t, z)$ exists for all $(t, z) \in \Omega_T \setminus E$. Recall that $E$ satisfies (\ref{property of E}). Similarly, there exists a Borel subset $\tilde E \subset \Omega_T$ such that for all $j$, 
    \begin{align*}
        \int_{\tilde E \cap ((0, T) \times K)}dt \wedge (dd^c\psi_j)^n = \int_{\tilde E \cap ((0, T) \times K)}dt \wedge (dd^c\psi)^n = 0
    \end{align*}
    for all $j$, $\frac{\partial u_j}{\partial t}(t+\delta, z)$ exists for all $(t, z) \in \Omega_T \setminus \tilde E$, and
    \begin{align}\label{core inequality2}
        v_j(t+\delta, z) - v_j(t, z) \geq \delta \frac{\partial v_j}{\partial t}(t+\delta, z)
    \end{align}
    holds for all $(t, z) \in \Omega_T \setminus \tilde E$.
    \\ \indent First, using (\ref{core inequality}) and (\ref{core inequality2}), 
    \begin{align*}
        I_1 &= C_1\int_{(J \times K)\setminus E}\chi \left\vert \frac{\partial v_j}{\partial t}(t, z) - \frac{C_0}{t} - \frac{u_j(t+\delta, z) - u_j(t, z)}{\delta}\right\vert dt \wedge (dd^c\psi_j)^n \\
        &\leq C_1\int_{(J \times K) \setminus E}\chi\left(\frac{\partial v_j}{\partial t}(t, z) - \frac{v_j(t+\delta, z) - v_j(t, z)}{\delta}\right)dt \wedge (dd^c\psi_j)^n\\
        &\quad + C_1C_0\int_{(J \times K)\setminus E}\chi \left\vert \frac{1}{t}-\frac{\log(t+\delta) - \log t}{\delta}\right\vert dt \wedge  (dd^c\psi_j)^n \\
        &\leq C_1\int_{(J \times K) \setminus (E \cap \tilde E)}\chi \left(\frac{\partial v_j}{\partial t}(t, z) - \frac{\partial v_j}{\partial t}(t+\delta, z)\right)dt \wedge (dd^c\psi_j)^n \\
        &\quad + C_1C_0\int_{J \times K}\chi \left\vert \frac{1}{t}-\frac{\log(t+\delta) - \log t}{\delta}\right\vert dt \wedge  (dd^c\psi_j)^n.
    \end{align*}
    By Lemma \ref{CLN}, there exists a constant $C_2 > 0$ such that
    \begin{align} \label{CLN1}
        \int_{J \times K}dt \wedge (dd^c\psi_j)^n \leq C_2.
    \end{align}
    Now we choose
    \begin{align} \label{h1}
        \delta = \min\left\{\frac{\varepsilon\inf_{t \in J}t^2}{8C_0C_1C_2\lVert \chi\rVert_{L^{\infty}(\Omega_T)}}, \frac{\varepsilon \inf_{t \in J}t}{8C_1C_2(\kappa_0+C_0)\left\Vert \frac{\partial \chi}{\partial t}\right\Vert_{L^{\infty}(\Omega_T)}}, T-\sup_{t \in J}t\right\}.
    \end{align}
    By (\ref{CLN1}) and (\ref{h1}),
    \begin{align*}
        &C_0C_1\int_{J \times K}\chi \left\vert \frac{1}{t}-\frac{\log(t+\delta) - \log t}{\delta}\right\vert dt \wedge  (dd^c\psi_j)^n \\&\quad \leq C_0C_1\int_{J \times K}\chi \left\vert \frac{1}{t}-\frac{1}{t+\delta}\right\vert dt \wedge (dd^c\psi_j)^n \\&\quad \leq C_0C_1\delta\int_{J \times K} \frac{\chi}{t^2} dt \wedge (dd^c\psi_j)^n \leq \frac{\varepsilon}{8}.
    \end{align*}
    Using (\ref{core inequality 3}), (\ref{core inequality}),
    (\ref{core inequality2}), (\ref{CLN1}), (\ref{h1}) and Lemma $3.1$, we get
    \begin{align*}
        &C_1\int_{(J \times K) \setminus (E \cap \tilde E)}\chi \left(\frac{\partial v_j}{\partial t}(t, z) - \frac{\partial v_j}{\partial t}(t+\delta, z)\right) dt \wedge (dd^c\psi_j)^n \\
        \quad &= C_1\int_{(J \times K) \setminus (E \cap \tilde E)}\frac{\partial \chi}{\partial t} (v_j(t+\delta, z) - v_j(t, z))dt \wedge (dd^c\psi_j)^n \\
        \quad &\leq C_1\delta\int_{(J \times K) \setminus (E \cap \tilde 
        E)}\left\vert \frac{\partial \chi}{\partial t}\right\vert \max\left(\left\vert \frac{\partial v_j}{\partial t}(t, z)\right\vert, \left\vert \frac{\partial v_j}{\partial t}(t+\delta, z)\right\vert\right)  dt \wedge (dd^c\psi_j)^n \\
        \quad &\leq C_1\delta\int_{(J \times K) \setminus (E \cap \tilde E)}\left\vert \frac{\partial \chi}{\partial t} \right\vert \left\vert \frac{\kappa_0+C_0}{t}\right\vert dt \wedge (dd^c\psi_j)^n \leq \frac{\varepsilon}{8}.
    \end{align*}
    Indeed, the first equality holds by Lemma $3.1$. The first inequality holds by (\ref{core inequality}) and (\ref{core inequality2}), and the second inequality holds by (\ref{core inequality 3}). The last inequality holds by (\ref{CLN1}) and (\ref{h1}).
    Hence for all $j$,
    \begin{align} \label{I1}
        I_1 \leq \frac{\varepsilon}{8}+\frac{\varepsilon}{8} = \frac{\varepsilon}{4}.
    \end{align}
    \indent Secondly, we estimate $I_2$ by the same argument for the estimate of $I_1$. For all $j$, we have
    \begin{align} \label{I2}
        I_2 = C_1\int_{(J \times K)\setminus G}\chi \left\vert \frac{\partial u}{\partial t}(t, z) - \frac{u(t+\delta, z) - u(t, z)}{\delta}\right\vert dt \wedge (dd^c\psi_j)^n \leq \frac{\varepsilon}{4}.
    \end{align}
    \indent Next, we estimate $I_3$ as follows.
    \begin{equation}\label{I3-1}
        \begin{aligned}
            I_3 &\leq \frac{C_1}{\delta}\int_{J \times K}\chi \lvert u_j(t+\delta, z) - u(t+\delta, z)\rvert dt \wedge (dd^c\psi_j)^n\\ &\quad + \frac{C_1}{\delta}\int_{J \times K}\chi \lvert u_j(t, z) - u(t, z)\rvert dt \wedge (dd^c\psi_j)^n.
        \end{aligned}
    \end{equation}
    Since $u_j(t, \cdot) \rightarrow u(t, \cdot)$ in capacity as $j \rightarrow \infty$ for a.e. $t \in (0, T)$, by \cite[Theorem 4.26]{GZ17}, there exists $j_1 > 0$ such that for all $j > j_1$, 
    \begin{equation}
        \begin{aligned} \label{I3-2}
            &\frac{C_1}{\delta}\int_{J \times K}\chi \lvert u_j(t+\delta, z) - u(t+\delta, z)\rvert dt \wedge (dd^c\psi_j)^n\\ &\quad + \frac{C_1}{\delta}\int_{J \times K}\chi \lvert u_j(t, z) - u(t, z)\rvert dt \wedge (dd^c\psi_j)^n \leq \frac{\varepsilon}{4}.
        \end{aligned}
    \end{equation}
    Therefore by (\ref{I3-1}) and (\ref{I3-2}), for all $j > j_1$,
    \begin{align}\label{I3}
        I_3 \leq \frac{\varepsilon}{4}.
    \end{align}
    \indent Again by the convergence of $u_j(t, \cdot) \rightarrow u(t, \cdot)$ in capacity for a.e. $t \in (0, T)$ and continuity of $F$, there exists $j_2 > 0$ such that for all $j > j_2$, 
    \begin{align} \label{I4}
        I_4 \leq \frac{\varepsilon}{4}.
    \end{align}
    \indent Combining the estimates of $I_1$, $I_2$, $I_3$ and $I_4$ in (\ref{I1}), (\ref{I2}), (\ref{I3}) and (\ref{I4}), we get
    \begin{align*}
        \left\vert \int_{\Omega_T}\chi \left(e^{\frac{\partial u_j}{\partial t}+F(t, z, u_j)}-e^{\frac{\partial u}{\partial t}+F(t, z, u)}\right) dt \wedge (dd^c\psi_j)^n \right\vert \leq \varepsilon
    \end{align*}
    for all $j > \max\{j_1, j_2\}$, which implies (\ref{R1}). \\
    \indent Now, we show that
    \begin{align}\label{R2}
        e^{\frac{\partial u}{\partial t}+F(t, z, u)}dt \wedge (dd^c\psi_j)^n \rightarrow e^{\frac{\partial u}{\partial t}+F(t, z, u)}dt \wedge (dd^c\psi)^n \text{ weakly }
    \end{align}
    as $j \rightarrow \infty$. Indeed, we use the similar arguments in the proof of (\ref{R1}). For $\delta$ chosen in (\ref{h1}), we have
    \begin{align*}
        &\left\vert \int_{\Omega_T}\chi e^{\frac{\partial u}{\partial t}+F(t, z, u)}dt \wedge \{(dd^c\psi_j)^n - (dd^c\psi)^n\}\right\vert \\
        &\quad\leq \left\vert \int_{J \times K}\chi \left(e^{\frac{\partial u}{\partial t}+F(t, z, u)} - e^{\frac{u(t+\delta, z) - u(t, z)}{\delta} + F(t, z, u)}\right) dt \wedge (dd^c\psi_j)^n\right\vert \\
        &\quad \quad + \left\vert \int_{J \times K}\chi \left(e^{\frac{\partial u}{\partial t}+F(t, z, u)} - e^{\frac{u(t+\delta, z) - u(t, z)}{\delta} + F(t, z, u)}\right) dt \wedge (dd^c\psi)^n\right\vert \\
        &\quad \quad + \left\vert \int_{J \times K}\chi e^{\frac{u(t+\delta, z) - u(t, z)}{\delta}+F(t, z, u)}dt \wedge \{(dd^c\psi_j)^n - (dd^c\psi)^n\}\right\vert \\
        &\quad =: I_5 + I_6 + I_7.
    \end{align*}
    \indent We start with the estimate for $I_5$.
    \begin{align*}
        I_5 &= \left\vert \int_{J \times K}\chi e^{F(t, z, u)}\left(e^{\frac{\partial u}{\partial t}} - e^{\frac{u(t+\delta, z) - u(t, z)}{\delta}}\right) dt \wedge (dd^c\psi_j)^n \right\vert\\
        &\leq e^{M_F}\int_{J \times K}\chi\left\vert e^{\frac{\partial u}{\partial t}} - e^{\frac{u(t+\delta, z) - u(t, z)}{\delta}}\right\vert dt \wedge (dd^c\psi_j)^n\\
        &\leq C_3\int_{(J \times K)\setminus G}\chi\left\vert \frac{\partial u}{\partial t} - \frac{u(t+\delta, z) - u(t, z)}{\delta}\right\vert dt \wedge (dd^c\psi_j)^n,
    \end{align*}
    where $C_3 = \sup_{t \in J}e^{\frac{\kappa_0}{t}+\frac{2\lVert u\rVert_{L^{\infty}(\Omega_T)}}{\delta}+M_F} > 0$.
    By repeating the proof for the estimate of $I_1$, one can show that for all $j$,
    \begin{align}\label{I5}
        I_5 \leq \frac{\varepsilon}{4}.
    \end{align}
    \indent Similarly, for all $j$,
    \begin{align}\label{I6}
        I_6 \leq \frac{\varepsilon}{4}. 
    \end{align}
    \indent Finally, we estimate $I_7$ as follows. 
    \begin{align*}
        &I_7 = \left\vert \int_{J \times K}\chi(t, z)e^{\frac{u(t+\delta, z) - u(t, z)}{\delta} + F(t, z, u)}dt \wedge \{(dd^c\psi_j)^n - (dd^c\psi)^n\}\right\vert \\
        &\quad \leq \int_J dt \left\vert \int_{K}\chi(t, z)e^{\frac{u(t+\delta, z) - u(t, z)}{\delta}+F(t, z, u)}\{(dd^c\psi_j)^n - (dd^c\psi)^n\}\right\vert. 
    \end{align*}
    Let $\Phi_j(t) := \left\vert \int_{K}\chi(t, z)e^{\frac{u(t+\delta, z) - u(t, z)}{\delta}+F(t, z, u)}\{(dd^c\psi_j)^n - (dd^c\psi)^n\}\right\vert$.
    By the definition, $\Phi_j(t) \geq 0$ for all $t \in J$ and
    \begin{equation}
        \begin{aligned}\label{bdd}
            \Phi_j(t) &\leq \int_{K}\chi(t, z)e^{\frac{u(t+\delta, z) - u(t, z)}{\delta}+F(t, z, u)}(dd^c\psi_j)^n \\&\quad + \int_{K}\chi(t, z)e^{\frac{u(t+\delta, z) - u(t, z)}{\delta}+F(t, z, u)}(dd^c\psi)^n \\
        &\leq e^{\frac{2\lVert u\rVert_{L^{\infty}(\Omega_T)}}{\delta}+M_F}\lVert \chi\rVert_{L^{\infty}(\Omega_T)}\left(\int_{K}(dd^c\psi_j)^n+\int_{K}(dd^c\psi)^n\right) \\
         &\leq 2C_4e^{\frac{2\lVert u\rVert_{L^{\infty}(\Omega_T)}}{\delta}+M_F}\lVert \chi\rVert_{L^{\infty}(\Omega_T)}
        \end{aligned}
    \end{equation}
    for some uniform constant $C_4$ by the elliptic Chern-Levine-Nirenberg inequality. Note that for each $t \in (0, T)$, the map $z \mapsto \frac{u(t+\delta, z) - u(t, z)}{\delta}+F(t, z, u)$ is quasi-continuous in $\Omega$ since $u(t+\delta, \cdot)$, $u(t, \cdot) \in PSH(\Omega)$ and $F$ is continuous. By \cite[Lemma 4.21]{GZ17}, for every $t \in (0, T)$ we have
    \begin{equation}\label{con}
        \begin{aligned}
            \Phi_j(t) = \left\vert \int_{\Omega}\chi(t, \cdot)e^{\frac{u(t+\delta, z) - u(t, z)}{\delta}+F(t, z, u)}\{(dd^c\psi_j)^n - (dd^c\psi)^n\}\right\vert \rightarrow 0
        \end{aligned}
    \end{equation}
    as $j \rightarrow \infty$. It follows from (\ref{bdd}) and (\ref{con}) that there exists $j_3 > 0$ such that for all $j > j_3$,
    \begin{align}\label{I7}
        I_7 \leq \int_J \Phi_j(t)dt \leq \frac{\varepsilon}{2}.
    \end{align}
    \indent Combining the estimates for $I_5$, $I_6$ and $I_7$ in (\ref{I5}), (\ref{I6}) and (\ref{I7}), we have
    \begin{align*}
        \left\vert \int_{\Omega_T}\chi e^{\frac{\partial u}{\partial t}+F(t, z, u)}dt \wedge \{(dd^c\psi_j)^n - (dd^c\psi)^n\} \right\vert \leq \varepsilon
    \end{align*}
    for all $j > j_3$, which implies (\ref{R2}). \\
    \indent Finally, it follows from (\ref{R1}) and (\ref{R2}) that
    \begin{align*}
        &\int_{\Omega_T}\chi e^{\frac{\partial u_j}{\partial t}+F(t, z, u_j)}dt \wedge (dd^c\psi_j)^n \rightarrow \int_{\Omega_T}\chi e^{\frac{\partial u}{\partial t}+F(t, z, u)}dt \wedge (dd^c\psi)^n 
    \end{align*}
    as $j \rightarrow \infty$. 
\end{proof}
\section{Bounded Subsolution Theorem}
\subsection{Compactly Supported Measures}
We first observe the relationship between complex Monge-Amp\`ere flows and complex Monge-Amp\`ere equations.
\begin{lemma}
    Let $h_1 \in C((0, T) \times \partial \Omega)$. Assume that there exists a function $v \in \mathcal{P}(\Omega_T) \cap L^{\infty}(\Omega_T)$ satisfying
    \begin{align*}
        \begin{cases}
            &dt \wedge (dd^cv)^n \geq e^{\frac{\partial v}{\partial t}+F(t, z, v)}dt \wedge d\mu \text{ in } \Omega_T, \\
            &\lim_{(t, z) \rightarrow (\tau, \zeta)}v(t, z) = h_1(\tau, \zeta) \text{ for all } (\tau, \zeta) \in (0, T) \times \partial \Omega.
        \end{cases}
    \end{align*}
    If $\mu$ is compactly supported or $h_1 \equiv 0$, then there exists a function $\varphi \in PSH(\Omega) \cap L^{\infty}(\Omega)$ satisfying
    \begin{align}\label{4-1-0}
        \begin{cases}
            &(dd^c\varphi)^n \geq d\mu \text{ in } \Omega, \\
            &\lim_{z \rightarrow \partial \Omega}\varphi(z) = 0.
        \end{cases}
    \end{align}
    \begin{proof}
        Fix $t_0 \in (0, T)$ satisfying
        \begin{align}\label{4-1-1}
            (dd^cv(t_0, \cdot))^n \geq e^{\frac{\partial v}{\partial t}(t_0, \cdot)+F(t_0, \cdot, v)}d\mu.
        \end{align}
        Indeed, such $t_0$ exists by \cite[Proposition 3.2]{GLZ21-2}. Since $v \in \mathcal{P}(\Omega_T) \cap L^{\infty}(\Omega_T)$, 
        \begin{align}\label{4-1-2}
            \sup_{z \in \Omega}\left\vert \limsup_{s \rightarrow 0}\frac{v(t_0+s, z) - v(t_0, z)}{s} \right\vert - \inf_{z \in \Omega}F(t_0, z, -\lVert v\rVert_{L^{\infty}(\Omega_T)})\leq M
        \end{align}
        for some constant $M > 0$.
        Let $\varphi(z) := e^{\frac{M}{n}}v(t_0, z) \in PSH(\Omega) \cap L^{\infty}(\Omega)$. Combining (\ref{4-1-1}) and (\ref{4-1-2}), we have
        \begin{align*}
            \begin{cases}
                &(dd^c\varphi)^n = e^M(dd^cv(t_0, \cdot))^n \geq e^{M+\frac{\partial v}{\partial t}(t_0, \cdot)+F(t_0, \cdot, v)}d\mu \geq d\mu, \\
                &\lim_{z \rightarrow \zeta}\varphi(z) = \lim_{z \rightarrow \zeta} e^{\frac{M}{n}}v(t_0, z) = e^{\frac{M}{n}}h_1(t_0, \zeta) \text{ for all } \zeta \in \partial \Omega.
            \end{cases}
        \end{align*}
        If $h_1 \equiv 0$, the proof is done. If $\mu$ is compactly supported, one can modify $\varphi$ so that $(dd^c\varphi)^n \geq d\mu$ and $\lim_{z \rightarrow \partial \Omega}\varphi(z) = 0$.
    \end{proof}
\end{lemma}
\begin{remark}
    By \cite[Theorem A]{Ko95}, (\ref{4-1-0}) implies that there exists a function $\psi \in PSH(\Omega) \cap L^{\infty}(\Omega)$ satisfying
    \begin{align*}
    \begin{cases}
        &(dd^c\psi)^n = d\mu \text{ in } \Omega,\\
        &\lim_{z \rightarrow \partial \Omega}\psi(z) = 0.
    \end{cases}
    \end{align*}
\end{remark}
The result of Lemma $4.1$ infers that under the assumption of the compact support for the given measure, existence of a bounded plurisubharmonic function satisfying (\ref{4-1-0}) is a necessary condition for the solvability of (\ref{CD}). We will now prove that it is also a sufficient condition.
\begin{theorem}
    Assume that $\mu$ is compactly supported in $\Omega$. If there exists a function $\varphi \in PSH(\Omega) \cap L^{\infty}(\Omega)$ satisfying
    \begin{align}\label{4-2-1}
        \begin{cases}
            &(dd^c\varphi)^n \geq d\mu \text{ in } \Omega, \\
            &\lim_{z \rightarrow \partial \Omega}\varphi(z) = 0,
        \end{cases}
    \end{align}
    then there exists a function $u \in \mathcal{P}(\Omega_T) \cap L^{\infty}(\Omega_T)$ satisfying
    \begin{align*}
        \begin{cases}
            &dt \wedge (dd^cu)^n = e^{\frac{\partial u}{\partial t}+F(t, z, u)}dt \wedge d\mu \text{ in } \Omega_T, \\
            &\lim_{(t, z) \rightarrow (\tau, \zeta)}u(t, z) = h(\tau, \zeta) \text{ for all } (\tau, \zeta) \in [0, T) \times \partial \Omega,\\
            &\lim_{t \rightarrow 0^+}u(t, \cdot) = h_0 \text{ in } L^1(\Omega, \mu).
        \end{cases}
    \end{align*}
\end{theorem}
    \begin{proof}
        Assume that $\text{supp}\mu \Subset U \Subset \Omega$ for some open set $U$. Let $\Omega'$ be an open neighborhood of $\overline{\Omega}$. We choose $\rho$ a defining function of $\Omega$ which is smooth on $\Omega'$. We may assume that $\rho \leq \varphi$ on $\overline{U}$. Define
        \begin{align*}
            \hat{\varphi}(z) = \begin{cases}
                \max\{\varphi(z), \rho(z)\} &\text{ for } z \in \overline{\Omega}, \\
                \rho(z) &\text{ for } z \in \Omega' \setminus \Omega.
            \end{cases}
        \end{align*}
        Note that $\hat{\varphi}$ is a plurisubharmonic function on $\Omega'$ which satisfies (\ref{4-2-1}). From now, we replace $\varphi$ by $\hat{\varphi}$. \\
        \indent Let $\eta \geq 0$ be a smooth radial function on $\mathbb{C}^n$ with a compact support on a unit ball $\mathbb{B} \subset \mathbb{C}^n$ which satisfies $\int_{\mathbb{B}}\eta dV = 1$. Let $\eta_j(z) := j^{2n}\eta(jz)$ for each $j \in \mathbb{Z}^+$. Fix $j_0 \in \mathbb{Z}^+$ satisfying $\Omega \Subset \Omega_{j_0} \Subset \Omega'$, where $\Omega_{j_0} := \{z \in \Omega' ~\mid~ d(z, \partial\Omega') > 1/j_0\}$. For each $j > j_0$, we define $w_j(z) := (\eta_j * \varphi)(z)$ on $\Omega_{j_0}$, so that $w_j \downarrow \varphi$ on $\Omega_{j_0}$ as $j \rightarrow \infty$. From now, we always pick $j > j_0$. \\
        \indent Let $g_j := \det\left(\frac{\partial ^2 w_j}{\partial z_k \partial \overline{z}_l}\right)$ so that $(dd^cw_j)^n = g_jdV$ in $\Omega_{j_0}$ for each $j$. Let $d\nu := (dd^c\varphi)^n$ in $\Omega$. By Radon-Nikodym theorem, $d\mu = fd\nu$ for some $f \in L^1(\Omega, d\nu)$ satisfying $0 \leq f \leq 1$. Let us assume that $f$ is a continuous function with compact support in $\Omega$, as in the proof of \cite[Theorem 3.1]{KN23}. We will remove the continuity assumption on $f$ at the end of the proof. Note that $fg_j \in L^p(\Omega)$ for every $p > 1$. By \cite[Theorem B]{Ko98}, there exist functions $\psi_j \in PSH(\Omega) \cap C(\overline{\Omega})$ such that
        \begin{align*}
            \begin{cases}
                &(dd^c\psi_j)^n = fg_jdV \text{ in } \Omega, \\
                &\psi_j = 0 \text{ on } \partial \Omega.
            \end{cases}
        \end{align*}
        By the proof of \cite[Theorem 4.7]{Ko05}, $\psi := (\limsup_{j \rightarrow \infty}\psi_j)^*$ satisfies
        \begin{align*}
            \begin{cases}
                &(dd^c\psi)^n = d\mu \text{ in } \Omega, \\
                &\lim_{z \rightarrow \partial\Omega}\psi(z) = 0.
            \end{cases}
        \end{align*}
        Furthermore, $\psi_j \rightarrow \psi$ as $j \rightarrow \infty$ in capacity \cite[Corollary 2.15]{Xi08}. \\
        \indent Let us consider the following Cauchy-Dirichlet problem : 
        \begin{align}\label{4-2-2}
            \begin{cases}
                &u_j \in \mathcal{P}(\Omega_T) \cap C((0, T) \times \overline{\Omega}), \\
                &dt \wedge (dd^cu_j)^n = e^{\frac{\partial u_j}{\partial t}+F(t, z, u_j)}(fg_j+1/j)dt \wedge dV \text{ in } \Omega_T, \\
                &\lim_{(t, z) \rightarrow (\tau, \zeta)}u_j(t, z) = h(\tau, \zeta) \text{ for all } (\tau, \zeta) \in [0, T) \times \partial \Omega, \\
                &\lim_{t \rightarrow 0^+}u_j(t, z) = h_0(z) \text{ for all } z \in \Omega.
            \end{cases}
        \end{align}
        Since $G_j := fg_j+1/j >0$ on $\Omega$ and $G_j \in L^p(\Omega)$ for all $p \geq 1$, (\ref{4-2-2}) is solvable by \cite[Theorem 6.4]{GLZ21-2}. Again by \cite[Theorem B]{Ko98}, there exist functions $\Tilde{\psi}_j \in PSH(\Omega) \cap C(\overline{\Omega})$ satisfying
        \begin{align*}
            \begin{cases}
                &(dd^c\Tilde{\psi}_j)^n = G_jdV \text{ in } \Omega,\\
                &\Tilde{\psi}_j = 0 \text{ on } \partial \Omega.
            \end{cases}
        \end{align*}
        By Ko{\l}odziej's stability estimate (see e.g. \cite[Proposition 5.22]{GZ17}),
        \begin{align}\label{stability}
            \lVert \psi_j - \Tilde{\psi}_j\rVert_{L^{\infty}(\Omega)} \leq (\text{diam}(\Omega))^2 \times \left(\frac{1}{j}\right)^{1/n},
        \end{align}
        which implies that $\Tilde{\psi}_j \rightarrow \psi$ in capacity as $j \rightarrow \infty$. Since $w_j \downarrow \varphi$ as $j \rightarrow \infty$,
        \begin{align}\label{Linfinitynorm}
            \varphi - \sup_{\partial \Omega}w_j \leq w_j - \sup_{\partial \Omega}w_j \leq \psi_j \leq 0
        \end{align}
        holds for all $j$ by the elliptic comparison principle (see e.g. \cite[Corollary 3.30]{GZ17}). Combining \cite[Lemma 3.6]{GLZ21-2}, (\ref{stability}) and (\ref{Linfinitynorm}),
        \begin{equation}\label{Linfinitynorm2}
            \begin{aligned}
                &e^{M_F/n}\left(\varphi - \sup_{\partial \Omega}w_j - (\text{diam}(\Omega))^2 \times \left(\frac{1}{j}\right)^{1/n}\right) + M_h \\ &\quad\leq e^{M_F/n}\Tilde{\psi}_j + M_h \leq u_j \leq M_h \text{ for all }j.
            \end{aligned}
        \end{equation}
        Since $w_j \rightarrow 0$ uniformly on $\partial \Omega$ as $j \rightarrow \infty$, (\ref{Linfinitynorm2}) implies that $\{u_j\}_{j = 1}^{\infty}$ is uniformly bounded. 
        \\
        \indent Pick a subsequence of $\{u_j\}_{j = 1}^{\infty}$ which converges in $L^1(\Omega_T)$. By an abuse of notation, we again denote it as $\{u_j\}_{j = 1}^{\infty}$. Let 
        \begin{align*}
            u(t, z) := \lim_{r \rightarrow 0}\sup_{\zeta \in B_r(z) \setminus \{z\}}(\limsup_{j \rightarrow \infty}u_j(t, \zeta)).    
        \end{align*}
        Note that since $u_j \rightarrow u$ in $L^1(\Omega_T)$, $u$ remains same even if we replace $\{u_j\}_{j = 1}^{\infty}$ by its subsequence. It follows from \cite[Theorem 4.2, Theorem 4.8]{GLZ21-2} and the uniform boundedness of $\{u_j\}_{j = 1}^{\infty}$ that there exist constants $\kappa_0 > 0$ and $C_0 > 0$ which do not depend on $j$ and satisfy
        \begin{align}\label{4-2-4}
            t\left\vert \frac{\partial u_j}{\partial t}(t, z)\right\vert \leq \kappa_0 \text{ for all } j
        \end{align}
        for all $(t, z) \in \Omega_T$ where $\frac{\partial u_j}{\partial_t}$ is well-defined and
        \begin{align}\label{4-2-5}
            \frac{\partial^2u_j}{\partial t^2}(t, z) \leq C_0t^{-2} \text{ for all j}
        \end{align}
        in the sense of distributions in $(0, T)$ for all $z \in \Omega$. 
        \\
        \indent We first show that $u$ satisfies the Dirichlet boundary condition. Here we follow the construction in the proof of \cite[Lemma 2.4]{Nc13}. Fix $\varepsilon > 0$ and $t_0 \in (0, T)$ such that
        \begin{align*}
            (dd^cu_j(t_0, \cdot))^n = e^{\frac{\partial u_j}{\partial t}(t_0, \cdot) + F(t_0, \cdot, u_j)}(dd^c\Tilde{\psi}_j)^n \text{ in } \Omega.
        \end{align*}
        We show that $\lim_{z \rightarrow \zeta}u(t_0, z) = h(t_0, \zeta)$ for all $\zeta \in \partial \Omega$. Since $h(t_0, \cdot)$ is continuous on $\partial \Omega$, there exists a function $\Tilde{h}_{t_0} \in C^2(\overline{\Omega})$ which is an approximant of the continuous extension of $h(t_0, \cdot)$ so that $\lvert h(t_0, \cdot) - \Tilde{h}_{t_0}\rvert \leq \varepsilon$ on $\partial \Omega$. Choose $A > 0$ large enough so that $\Tilde{h}_{t_0} + A\rho \in PSH(\Omega)$. Let $B = e^{(\kappa_0/t+M_F)/n}$ and $V_j(t, z) := \Tilde{h}_{t_0}(z)+A\rho(z)+B\Tilde{\psi}_j(z)$. Then
        \begin{align*}
            \begin{cases}
                &(dd^cV_j(t_0, \cdot))^n \geq B^n(dd^c\Tilde{\psi}_j)^n \geq e^{\frac{\partial u_j}{\partial t}(t_0, \cdot)+F(t_0, \cdot, u_j)}(dd^c\Tilde{\psi}_j)^n, \\
                &V_j(t_0, z) - \varepsilon \leq h(t_0, z) \text{ on } \partial \Omega. 
            \end{cases}
        \end{align*}
        By the elliptic comparison principle, $V_j(t_0, z) - \varepsilon \leq u_j(t_0, z)$ in $\Omega$. 
        \\
        \indent By (\ref{stability}), $\lVert \psi_j - \Tilde{\psi}_j\rVert_{L^{\infty}(\Omega)} \leq \varepsilon$ holds for sufficiently large $j$. By the elliptic comparison principle, $w_j - \sup_{z \in\partial \Omega}w_j(z) \leq \psi_j $ holds
        for all $j$. Hence 
        \begin{equation}\label{dirichlet}
            \begin{aligned}
                u_j(t_0, z) &\geq \Tilde{h}_{t_0}(z) + A\rho(z) + B\left(w_j(z)-\sup_{z \in \partial\Omega} w_j(z) - \varepsilon\right) -\varepsilon\\
                &\geq h(t_0, z) + A\rho(z) + B(\varphi(z) - 2\varepsilon)-2\varepsilon    
            \end{aligned}
        \end{equation}
        holds for sufficiently large $j$ in a small neighborhood of $\partial \Omega$. Therefore
        \begin{align*}
            h(t_0, z) + A\rho(z) + B\varphi(z) - (2B+2)\varepsilon \leq \liminf_{j \rightarrow \infty}u_j(t_0, z) \leq u(t_0, z)
        \end{align*}
        holds for $z$ in a small neighborhood of $\partial \Omega$. This implies that $\lim_{z \rightarrow \zeta}u(t_0, z) = h(t_0, \zeta)$ for all $\zeta \in \partial \Omega$. By  the local uniform Lipschitzness of $u$ and $h$ on $(0, T)$, $\lim_{(t, z) \rightarrow (t_0, \zeta)}u(t, z) = h(t_0, \zeta)$ holds for all $(t_0, \zeta) \in (0, T) \times \partial \Omega$. For $t_0 = 0$, we may use the uniform continuity of $h(t, z)$ on $[0, \epsilon] \times \partial \Omega$ for some small $\epsilon > 0$.
        \\
\indent Next, we prove that $dt \wedge (dd^cu)^n = e^{\frac{\partial u}{\partial t}+F(t, z, u)}dt \wedge d\mu$ in $\Omega_T$. First, we show that $dt \wedge (dd^cu_j)^n \rightarrow dt \wedge (dd^cu)^n$ weakly as $j \rightarrow \infty$. By Lemma \ref{continuity}, it suffices to show that for a.e. $t \in (0, T)$,
        \begin{align*}
            (dd^cu_j(t, \cdot))^n \rightarrow (dd^cu(t, \cdot))^n \text{ weakly as } j \rightarrow \infty.
        \end{align*}
        In fact, for a.e. $t \in (0, T)$, 
        \begin{align*}
            (dd^cu_j(t, \cdot))^n &= e^{\frac{\partial u_j}{\partial t}+F(t, z, u_j)}(fg_j+1/j)dV\\ &\leq e^{\frac{\kappa_0}{t}+M_F}(dd^c\Tilde{\psi}_j)^n = \left(dd^c\left(e^{\frac{\kappa_0}{nt}+\frac{M_F}{n}}\Tilde{\psi}_j\right)\right)^n.
        \end{align*}
        Fix such a $t$. Then $u_j(t, \cdot) \rightarrow u(t, \cdot)$ weakly as $j \rightarrow \infty$ and $e^{\frac{\kappa_0}{nt}+\frac{M_F}{n}}\Tilde{\psi}_j \rightarrow e^{\frac{\kappa_0}{nt}+\frac{M_F}{n}}\Tilde{\psi}$ in capacity as $j \rightarrow \infty$. Moreover, note that $A$ and $B$ in (\ref{dirichlet}) are independent with $j$, which implies that for fixed $t_0 \in (0, T)$, $\liminf_{z \rightarrow \partial\Omega}(u_j(t_0, z)-u(t_0, z)) \geq 0$ uniformly for all $j$. Therefore by \cite[Theorem 2.14]{Xi08}, $u_j(t, \cdot) \rightarrow u(t, \cdot)$ in capacity as $j \rightarrow \infty$. \\
        \indent Next, we show that $e^{\frac{\partial u}{\partial t}+F(t, z, u_j)}dt \wedge (dd^c\Tilde{\psi}_j)^n \rightarrow e^{\frac{\partial u}{\partial t}+F(t, z, u)}dt \wedge (dd^c\psi)^n$ weakly as $j \rightarrow \infty$. We have already shown that $u_j(t, \cdot) \rightarrow u(t, \cdot)$ as $j \rightarrow \infty$ in capacity. Thus it follows from (\ref{4-2-4}), (\ref{4-2-5}) and Lemma $3.3$ that
\begin{align*}
    e^{\frac{\partial u_j}{\partial t}+F(t, z, u_j)}dt \wedge (dd^c\Tilde{\psi}_j)^n \rightarrow e^{\frac{\partial u}{\partial t}+F(t, z, u)}dt \wedge (dd^c\psi)^n \text{ weakly}
\end{align*}
as $j \rightarrow \infty$. \\
 \indent Lastly, we need to show that $u$ satisfies the Cauchy boundary condition in $L^1(\Omega, d\mu)$. Fix $t_1 \in (0, T)$ and $\varepsilon > 0$. Since $h$ is uniformly continuous in $[0, t_1] \times \partial \Omega$, there exists $t_2 \in (0, t_1)$ such that
\begin{align*}
    h(0, z) \leq h(t, z) + \varepsilon \text{ for all } (t, z) \in [0, t_2] \times \partial \Omega.
\end{align*}
Following the construction in \cite[Lemma 3.8]{GLZ21-2}, we define
\begin{align*}
    v_j(t, z) := h_0(z) - \varepsilon + t(\Tilde{\psi}_j(z) - C) + n[t\log(t/T) - t]
\end{align*}
where $C := M_F - \min(n\log T, 0)$. Then 
\begin{align*}
    dt \wedge (dd^cv_j)^n \geq t^n dt \wedge (dd^c\Tilde{\psi}_j)^n \geq e^{\frac{\partial v_j}{\partial t}+F(t, z, v_j)}dt \wedge (dd^c\Tilde{\psi}_j)^n
\end{align*}
and $v_j^* \leq h$ on $\partial_0\Omega_{t_1}$, where $\Omega_{t_1} = (0, t_1) \times \Omega$. By \cite[Proposition 4.1]{GLZ21-2}, 
\begin{align*}
    v_j(t, z) \leq u_j(t, z) \text{ on } \Omega_{t_1}.
\end{align*}
Since $\Tilde{\psi}_j$ is uniformly bounded, 
\begin{align*}
    h_0(z) - \varepsilon + t(\lVert \Tilde{\psi}_j\rVert_{L^{\infty}(\Omega_T)} - C) + n[t\log(t/T) - t] \leq v_j(t, z) \leq u(t, z),
\end{align*}
so that $h_0(z) - \varepsilon \leq \lim_{t \rightarrow 0^+}u(t, z)$. Since this holds for every $\varepsilon > 0$, it implies that $h_0(z) \leq \lim_{t \rightarrow 0^+}u(t, z)$. \\
\indent Next, fix an open set $U \Subset \Omega$ and let $\chi$ be a positive continuous test function on $\Omega$. Since $\{u_j\}_{j = 1}^{\infty}$ is uniformly bounded, $\int_U (dd^cu_j(t, \cdot))^n \leq C$ for every $t \in [0, T)$, for some constant $C > 0$ which does not depend on $j$. By \cite[Lemma 3.10]{GLZ21-2}, there exists a uniform constant $A > 0$ such that
\begin{align*}
    \int_U \chi u_j(t, \cdot)(dd^c\Tilde{\psi}_j)^n \leq \int_U \chi h_0(dd^c\Tilde{\psi}_j)^n + At
\end{align*}
holds for all $t \in (0, T)$ and for all $j$. Letting $j \rightarrow \infty$, 
\begin{align}\label{4-2-21}
    \int_U \chi u(t, \cdot)(dd^c\psi)^n \leq \int_U \chi h_0(dd^c\psi)^n + At 
\end{align}
for a.e. $t \in (0, T)$. Indeed, for a.e. $t \in (0, T)$, $u_j(t, \cdot) \rightarrow u(t, \cdot)$ in capacity and $\tilde
\psi_j \rightarrow \psi$ also in capacity as $j \rightarrow \infty$. Therefore by \cite[Theorem 4.26]{GZ17}, we have (\ref{4-2-21}).\\
\indent Since $u(t, z)$ is bounded, there exist a sequence of positive numbers $\{t_k\}_{k = 1}^{\infty}$ and a function $w \in PSH(\Omega)$ such that $t_k \rightarrow 0$ as $k \rightarrow \infty$ and for a.e. $z \in \Omega$ with respect to the Lebesgue measure, $u(t_k, z) \rightarrow w(z)$ as $k \rightarrow \infty$. One will show that for $\mu$-a.e. $z \in U$, $u(t_k, z) \rightarrow w(z)$ as $k \rightarrow \infty$. Indeed, we have $\sup_{k}\lVert u(t_k, \cdot)\rVert_{L^2(U, (dd^c\psi)^n)} < \infty$. Hence there exists a function $W \in L^2(U, (dd^c\psi)^n)$ such that by passing to a subsequence, $u(t_k, \cdot) \rightarrow W$ weakly as $k \rightarrow \infty$. Then by \cite[Lemma 2.1]{KN23}, 
\begin{align*}
    \lim_{k \rightarrow \infty}\int_U  u(t_{k}, \cdot)(dd^c\psi)^n = \int_U W(dd^c\psi)^n = \int_U  w(dd^c\psi)^n
\end{align*}
again by passing to a subsequence. Let us denote this subsequence by $\{t_k\}_{k = 1}^{\infty}$. Then since $\psi$ is bounded,
\begin{align}\label{l1convergence}
    \lim_{k \rightarrow \infty}\int_U \lvert u(t_k, \cdot) - w\rvert (dd^c\psi)^n = 0
\end{align}
by \cite[Corollary 2.2]{KN23}. Combining (\ref{4-2-21}) and (\ref{l1convergence}), we have $w(z) \leq h_0(z)$ for $\mu$-a.e. $z \in U$. Since $U$ is chosen arbitrarily, $w \leq h_0$ $\mu$-a.e. in $\Omega$. Therefore $\lim_{t \rightarrow 0^+}u(t, \cdot) = h_0$ in $L^1(\Omega, d\mu)$.\\
\indent Finally, we remove the continuity assumption on $f$. Since the Radon-Nikodym derivative $f$ is in $L^1(d\nu)$, there exists a sequence $\{f_j\}_{j = 1}^{\infty}$ of functions in $C_c(\Omega)$ such that
\begin{align*}
    f_j(dd^c\varphi)^n \rightarrow f(dd^c\varphi)^n \text{ weakly }
\end{align*}
as $j \rightarrow \infty$. For each $j$, we get a solution $\Tilde{u}_j \in \mathcal{P}(\Omega_T) \cap L^{\infty}(\Omega_T)$ for the Cauchy-Dirichlet problem 
\begin{align*}
    \begin{cases}
        &dt \wedge (dd^c\Tilde{u}_j)^n = e^{\frac{\partial \Tilde{u}_j}{\partial t}+F(t, z, \Tilde{u}_j)}f_jdt \wedge (dd^c\varphi)^n, \\
        &\lim_{(t, z) \rightarrow (\tau, \zeta)}\Tilde{u}_j(t, z) = h(\tau, \zeta) \text{ for all } (\tau, \zeta) \in [0, T) \times \partial \Omega, \\
        &\lim_{t \rightarrow 0^+}u(t, \cdot) = h_0 \text{ in } L^1(\Omega, f_j(dd^c\varphi)^n),
    \end{cases}
\end{align*} by the above arguments. Note that for each $j$, there exist $\phi_j, \phi \in PSH(\Omega) \cap L^{\infty}(\Omega)$ satisfying
\begin{align*}
    \begin{cases}
        &(dd^c\phi_j)^n = f_j(dd^c\varphi)^n  \text{ in } \Omega,\\
        &\lim_{z \rightarrow \partial\Omega}\phi_j(z) = 0,
    \end{cases}
\end{align*}
and 
\begin{align*}
    \begin{cases}
        &(dd^c\phi)^n = f(dd^c\varphi)^n = d\mu \text{ in } \Omega,\\
        &\lim_{z \rightarrow \partial \Omega}\phi(z) = 0,
    \end{cases}
\end{align*}
respectively. Since $(dd^c\phi_j)^n \leq (dd^c\varphi)^n$ for all $j$, $\varphi \leq \phi_j \leq 0$ and $\phi_j \rightarrow \phi$ in capacity as $j \rightarrow \infty$. Hence we may repeat all the arguments again to show that by passing to a subsequence if necessary, the limit of the sequence $\{\Tilde{u}_j\}_{j = 1}^{\infty}$ converges to the desired solution. The local uniform Lipschitzness and the local uniform semi-concavity of the family $\{\Tilde{u}_j\}_{j = 1}^{\infty}$ on $(0, T)$ follows from the construction of each $\Tilde{u}_j$.
\end{proof}
\indent Combining Lemma $4.1$ and Theorem $4.3$, we get the following corollary which is the parabolic version of \cite[Theorem A]{Ko95} for compactly supported measures on the right-hand side.
\begin{corollary}
    Assume that $\mu$ is compactly supported in $\Omega$. Assume that there exists a function $v \in \mathcal{P}(\Omega_T) \cap L^{\infty}(\Omega_T)$ satisfying
    \begin{align*}
        \begin{cases}
            &dt \wedge (dd^cv)^n \geq e^{\frac{\partial v}{\partial t}+F(t, z, v)}dt \wedge d\mu \text{ in } \Omega_T, \\
            &\lim_{(t, z) \rightarrow (\tau, \zeta)}v(t, z) = h(\tau, \zeta) \text{ for all } (\tau, \zeta) \in (0, T) \times \partial \Omega.
        \end{cases}
    \end{align*}
    Then there exists a function $u \in \mathcal{P}(\Omega_T) \cap L^{\infty}(\Omega_T)$ satisfying
    \begin{align*}
        \begin{cases}
            &dt \wedge (dd^cu)^n = e^{\frac{\partial u}{\partial t}+F(t, z, u)}dt \wedge d\mu \text{ in } \Omega_T, \\
            &\lim_{(t, z) \rightarrow (\tau, \zeta)}u(t, z) = h(\tau, \zeta) \text{ for all } (\tau, \zeta) \in [0, T) \times \partial \Omega, \\
            &\lim_{t \rightarrow 0^+}u(t, \cdot) = h_0 \text{ in } L^1(\Omega, \mu).
        \end{cases}
    \end{align*}
\end{corollary}
\subsection{Measures with a general support}
In the proof of the bounded subsolution theorem for the complex Monge-Amp\`ere equation \cite{Ko95, Ko05} and for the complex Hessian equation \cite{Nc13}, both authors removed the compact support assumption by using the comparison principle for the Monge-Amp\`ere equations and the Hessian equations, respectively. Even though we have a comparison principle \cite[Theorem 6.6]{GLZ21-2} for some complex Monge-Amp\`ere flows, this theorem holds only when the given measure has a strictly positive $L^p$ density function for some $p > 1$. Thus we need to use the approximating sequence of measures in the proof of Theorem 4.3 to apply this comparison principle.
\begin{theorem}
    If there exists a function $\varphi \in PSH(\Omega) \cap L^{\infty}(\Omega)$ satisfying
    \begin{align*}
        \begin{cases}
            &(dd^c\varphi)^n \geq d\mu \text{ in } \Omega, \\
            &\lim_{z \rightarrow \partial\Omega}\varphi(z) = 0,
        \end{cases}
    \end{align*}
    then there exists a function $u \in \mathcal{P}(\Omega_T) \cap L^{\infty}(\Omega_T)$ satisfying
    \begin{align*}
        \begin{cases}
            &dt \wedge (dd^cu)^n = e^{\frac{\partial u}{\partial t}+F(t, z, u)}dt \wedge d\mu \text{ in } \Omega_T, \\
            &\lim_{(t, z) \rightarrow (\tau, \zeta)}u(t, z) = h(\tau, \zeta) \text{ for all } (\tau, \zeta) \in [0, T) \times \partial \Omega, \\
            &\lim_{t \rightarrow 0^+}u(t, \cdot) = h_0 \text{ in } L^1(\Omega, \mu).
        \end{cases}
    \end{align*}
\end{theorem}
\begin{proof}
    By \cite[Theorem A]{Ko95}, there exists a function $\psi \in PSH(\Omega) \cap L^{\infty}(\Omega)$ satisfying
    \begin{align*}
        \begin{cases}
            &(dd^c\psi)^n = d\mu \text{ in } \Omega, \\
            &\lim_{z \rightarrow \partial \Omega}\psi(z) = 0.
        \end{cases}
    \end{align*}
    Let $\{\chi_k\}_{k = 1}^{\infty}$ be a sequence of smooth cut-off functions on $\Omega$ such that $\chi_k \uparrow 1$ as $k \rightarrow \infty$. Let $d\mu_k := \chi_kd\mu$. Since $(dd^c\varphi)^n \geq d\mu_k$ for all $k$, again by \cite[Theorem A]{Ko95}, there exists a sequence $\{\psi_k\}_{k = 1}^{\infty}$ of functions in $PSH(\Omega) \cap L^{\infty}(\Omega)$ such that
    \begin{align*}
        \begin{cases}
            &(dd^c\psi_k)^n = d\mu_k \text{ in } \Omega, \\
            &\lim_{z \rightarrow \partial \Omega}\psi_k(z) = 0.
        \end{cases}
    \end{align*}
    Since each $d\mu_k$ is compactly supported, we proceed as in the proof of Theorem $4.3$. \\
    \indent Let $\{U_k\}_{k = 1}^{\infty}$ be a sequence of open subsets of $\Omega$ such that
    \begin{align*}
        supp(\chi_k) \Subset U_k \Subset supp(\chi_{k+1}) \Subset U_{k+1} \Subset \cdots \Subset \Omega.
    \end{align*}
    Let $\Omega'$ be an open neighborhood of $\overline{\Omega}$. For each $k$, we choose $\rho_k$ a defining function of $\Omega$ which is smooth on $\Omega'$ and $\rho_k \leq \varphi$ on $\overline{U_k}$. Define
    \begin{align*}
        \hat{\varphi}_k(z) = \begin{cases}
            \max\{\varphi(z), \rho_k(z)\} &\text{ for } z \in \overline{\Omega}, \\
            \rho_k(z) &\text{ for } z \in \Omega' \setminus \Omega.
        \end{cases}
    \end{align*}
    \indent Let $\eta$ be the smooth radial function on $\mathbb{C}^n$ as defined in the proof of Theorem $4.3$. Fix $j_0 \in \mathbb{Z}^+$ satisfying $\Omega \Subset \Omega_{j_0} \Subset \Omega'$, where $\Omega_{j_0} = \{z \in \Omega' \mid d(z, \partial\Omega') > 1/j_0\}$. For each $j$ and $k$, let $w_{k, j}(z) := (\eta_j * \hat{\varphi}_k)(z)$ so that $w_{k, j} \downarrow \hat{\varphi}_k$ on $\Omega_{j_0}$ as $j \rightarrow \infty$. \\
    \indent Let $g_{k, j} := \det\left(\frac{\partial^2w_{k, j}}{\partial z_s \partial\overline{z}_{\ell}}\right)$ so that $(dd^cw_{k, j})^n = g_{k, j}dV$ in $\Omega_{j_0}$ for each $j$. By the Radon-Nikodym theorem, $d\mu = f(dd^c\varphi)^n$ for some $f \in L^1(\Omega, (dd^c\varphi)^n)$ with $ 0 \leq f \leq 1$. As in the proof of Theorem $4.3$, we assume that $f \in C_c(\Omega)$. Let us consider the following Cauchy-Dirichlet problem:
    \begin{align}\label{4-3-1}
        \begin{cases}
            &u_{k, j} \in \mathcal{P}(\Omega_T) \cap C((0, T) \times \overline{\Omega}), \\
            &dt \wedge (dd^cu_{k, j})^n = e^{\frac{\partial u_{k, j}}{\partial t} + F(t, z, u_{k, j})}(\chi_kfg_{k, j}+1/j)dt \wedge dV \text{ in } \Omega_T, \\
            &\lim_{(t, z) \rightarrow (\tau, \zeta)}u_{k, j}(t, z) = h(\tau, \zeta) \text{ for all } (\tau, \zeta) \in [0, T) \times \partial \Omega, \\
            &\lim_{t \rightarrow 0^+}u_{k, j}(t, z) = h_0(z) \text{ for all } z \in \Omega.
        \end{cases}
    \end{align}
    By \cite[Theorem 6.4]{GLZ21-2}, (\ref{4-3-1}) is solvable. For sufficiently large $j$, 
    \begin{align*}
        \chi_{k}fg_{k, j}dV &= \chi_{k}f\det\left(\frac{\partial^2 w_{k, j}}{\partial z_s \partial \overline{z}_{\ell}}\right)dV \\
        &= \chi_kf \det\left(\frac{\partial^2(\eta_j * \hat{\varphi}_k)}{\partial z_s \partial \overline{z}_{\ell}}\right)dV \\
        &= \chi_k f\det\left(\frac{\partial^2(\eta_j * \varphi)}{\partial z_s \partial \overline{z}_{\ell}}\right)dV
    \end{align*}
    since $\hat{\varphi}_k = \varphi$ on $U_k \supset supp(\mu_k)$. Therefore if $k_1 \leq k_2$, for sufficiently large $j$, 
    \begin{align*}
        \chi_{k_1}fg_{k_1, j}dV &= \chi_{k_1}f\det\left(\frac{\partial^2(\eta_j * \varphi)}{\partial z_s \partial \overline{z}_{\ell}}\right)dV \\
        &\leq \chi_{k_2}f\det\left(\frac{\partial^2(\eta_j * \varphi)}{\partial z_s \partial \overline{z}_{\ell}}\right) = \chi_{k_2}fg_{k_2, j}dV.
    \end{align*}
    Hence by \cite[Theorem 6.6]{GLZ21-2}, $u_{k_1, j} \geq u_{k_2, j}$. Fix $k_1$. By Theorem $4.3$, there exists a sequence $\{u_{k_1, j_i}\}_{i =1}^{\infty}$ such that
    \begin{align*}
        u_{k_1}(t, z) := \lim_{r \rightarrow 0}\sup_{\zeta \in B_r(z) \setminus \{z\}}(\limsup_{i \rightarrow \infty}u_{k_1, j_i}(t, \zeta))
    \end{align*}
    satisfies
    \begin{align*}
        \begin{cases}
            &u_{k_1} \in \mathcal{P}(\Omega_T) \cap L^{\infty}(\Omega_T), \\
            &dt \wedge (dd^cu_{k_1})^n = e^{\frac{\partial u_{k_1}}{\partial t}+F(t, z, u_{k_1})}dt \wedge d\mu_{k_1} \text{ in } \Omega_T, \\
            &\lim_{(t, z) \rightarrow (\tau, \zeta)}u_{k_1}(t, z) = h(\tau, \zeta) \text{ for all } (\tau, \zeta) \in [0, T) \times \partial \Omega,\\
            &\lim_{t \rightarrow 0^+}u_{k_1}(t, \cdot) = h_0(\cdot) \text{ in } L^1(\Omega, d\mu_{k_1}).
        \end{cases}
    \end{align*}
    Pick $k_2 \geq k_1$. Since $\{u_{k_2, j_i}\}_{i = 1}^{\infty}$ is uniformly bounded, by passing to a subsequence, $u_{k_2, j_i}$ converges to some function in $L^1(\Omega_T)$. Let us denote this subsequence by $\{u_{k_2, s_i}\}_{i = 1}^{\infty}$. Let $u_{k_2}(t,z) := \lim_{r \rightarrow 0}\sup_{\zeta \in B_r(z) \setminus \{z\}}(\limsup_{i \rightarrow \infty} u_{k_2, s_i}(t, \zeta))$. Again by Theorem $4.3$, $u_{k_2}$ satisfies
    \begin{align*}
        \begin{cases}
            &u_{k_2} \in \mathcal{P}(\Omega_T) \cap L^{\infty}(\Omega_T), \\
            &dt \wedge (dd^cu_{k_2})^n = e^{\frac{\partial u_{k_2}}{\partial t}+F(t, z, u_{k_2})}dt \wedge d\mu_{k_2} \text{ in } \Omega_T, \\
            &\lim_{(t, z) \rightarrow (\tau, \zeta)}u_{k_2}(t, z) = h(\tau, \zeta) \text{ for all } (\tau, \zeta) \in [0, T) \times \partial \Omega,\\
            &\lim_{t \rightarrow 0^+}u_{k_2}(t, \cdot) = h_0(\cdot) \text{ in } L^1(\Omega, d\mu_{k_2}).
        \end{cases}
    \end{align*}
    Since $u_{k_2, s_i} \leq u_{k_1, s_i}$ for each $i$, $u_{k_2} \leq u_{k_1}$. By repeating this argument, we get a sequence $\{u_{k_i}\}_{i = 1}^{\infty}$ such that $u_{k_i}$ satisfies
    \begin{align*}
        \begin{cases}
            &u_{k_i} \in \mathcal{P}(\Omega_T) \cap L^{\infty}(\Omega_T), \\
            &dt \wedge (dd^cu_{k_i})^n = e^{\frac{\partial u_{k_i}}{\partial t}+F(t, z, u_{k_i})}dt \wedge d\mu_{k_i} \text{ in } \Omega_T, \\
            &\lim_{(t, z) \rightarrow (\tau, \zeta)}u_{k_i}(t, z) = h(\tau, \zeta) \text{ for all } (\tau, \zeta) \in [0, T) \times \partial \Omega,\\
            &\lim_{t \rightarrow 0^+}u_{k_2}(t, \cdot) = h_0(\cdot) \text{ in } L^1(\Omega, d\mu_{k_i}),
        \end{cases}
    \end{align*}
    for each $i$. Moreover, $\{u_{k_i}\}_{i = 1}^{\infty}$ is a decreasing sequence. For simplicity, we denote this sequence by $\{u_k\}_{k = 1}^{\infty}$. By the uniform boundedness of $\{\psi_k\}_{k = 1}^{\infty}$, $\{u_k\}_{k = 1}^{\infty}$ is also uniformly bounded. Therefore $u_k \downarrow u$ for some $u \in \mathcal{P}(\Omega_T) \cap L^{\infty}(\Omega_T)$ and by Lemma $2.3$,
    \begin{align*}
        dt \wedge (dd^cu_k)^n \rightarrow dt \wedge (dd^cu)^n \text{ weakly as } k \rightarrow \infty.
    \end{align*}
    \indent By using the similar argument of the proof of Theorem $4.3$, one can show that $u$ satisfies all the boundary conditions. By \cite[Proposition 2.9]{GLZ21-2},  
    \begin{align*}
        e^{\frac{\partial u_k}{\partial t}+F(t, z, u_k)}dt \wedge d\mu_k \rightarrow e^{\frac{\partial u}{\partial t}+F(t, z ,u)}dt \wedge d\mu \text{ weakly as } k \rightarrow \infty.
    \end{align*}
    \indent Finally, let us consider the case when $f \in L^1(\Omega, (dd^c\varphi)^n)$. As we have done in the last part of the proof for Theorem 4.3, one can find a sequence $\{f_j\}_{j = 1}^{\infty}$ of functions in $C_c(\Omega)$ such that
    \begin{align*}
        f_j(dd^c\varphi)^n \rightarrow f(dd^c\varphi)^n \text{ weakly as } j \rightarrow \infty. 
    \end{align*}
    By arguing as in the final part of the proof of Theorem 4.3, we get the desired solution.
\end{proof}
\section{Cauchy Boundary Condition}
\subsection{Counterexample} We construct a counterexample which shows that the Cauchy-Dirichlet problem with the pointwise Cauchy boundary condition is not solvable in general.
\begin{lemma}
    The Cauchy-Dirichlet problem (\ref{CD}) with the boundary condition that 
    \begin{align*}
        \lim_{t \rightarrow 0^+}u(t, z) = h_0(z) \text{ for all } z \in \Omega
    \end{align*}
    is not solvable in general.
\end{lemma}
\begin{proof}
    Let us pick $\mu = 0$ and consider the homogeneous equation:
		\begin{align}\label{homogeneous}
			\begin{cases}
				&u \in \mathcal{P}(\Omega_T) \cap L^{\infty}(\Omega_T), \\
				&dt \wedge (dd^cu)^n = 0, \\
				&\lim_{(t, z) \rightarrow (\tau, \zeta)} = h(\tau, \zeta) \text{ for all } (\tau, \zeta) \in [0, T) \times \partial \Omega, \\
				&\lim_{t \rightarrow 0^+}u(t, z) = h_0(z) \text{ for all } z \in \Omega.
			\end{cases}
		\end{align}
	If we pick $h$ a function on $\partial_0\Omega_T$ defined by
	\begin{align*}
		h(t, z) = \begin{cases}
			0 &\text{ if } (t, z) \in [0, T) \times \partial\Omega, \\
			\phi(z) &\text{ if } (t, z)  \in \{0\} \times \Omega,
		\end{cases}
	\end{align*}
	where $\phi(z)$ is a smooth defining function of $\Omega$ which is strictly plurisubharmonic, then (\ref{homogeneous}) is not solvable. Indeed, assume that $u$ satisfies (\ref{homogeneous}). Fix a positive continuous test function $\chi$ on $\Omega_T$. Then 
	\begin{align*}
		\int_{\Omega_T}\chi dt \wedge (dd^cu)^n = \int_0^Tdt \int_{\Omega}\chi(t, \cdot)(dd^cu(t, \cdot))^n = 0.
	\end{align*}
	Let $S_{\chi}(t) := \int_{\Omega}\chi(t, \cdot)(dd^cu(t, \cdot))^n$. By Lemma $2.1$, $S_{\chi}(t)$ is continuous. Therefore
	\begin{align*}
		0 \leq \int_0^T S_{\chi}(t)dt = 0,
	\end{align*}
	which implies that $S_{\chi}(t) = 0$ for all $t \in (0, T)$. Since this holds for every positive continuous test function $\chi$, $(dd^cu(t, \cdot))^n = 0$ for all $t \in (0, T)$. This implies that for all $t \in (0, T)$, $u(t, \cdot) \in PSH(\Omega) \cap L^{\infty}(\Omega)$ satisfies
	\begin{align*}
		\begin{cases}
			&(dd^cu(t, \cdot))^n = 0 \text{ in } \Omega, \\
			&\lim_{z \rightarrow \partial \Omega}u(t, z) = 0,
		\end{cases}
	\end{align*}
	therefore $u(t, \cdot) = 0$. As it holds for every $t$, $u \equiv 0$. But it contradicts to the fact that
	\begin{align*}
		\lim_{t \rightarrow 0^+}u(t, z) = h_0(z) = \phi(z) \text{ for all } z \in \Omega. 
	\end{align*}
    This completes the proof. 
\end{proof}
\subsection{Admissible class} 
    While showing the Cauchy boundary condition on Theorems $4.3$ and $4.5$, we showed that $u$ satisfies
    \begin{align*}
        h_0(z) \leq \lim_{t \rightarrow 0^+}u(t, z) \text{ for all } z \in \Omega
    \end{align*}
    and
    \begin{align*}
        \lim_{t \rightarrow 0^+}u(t, z) \leq h_0(z) \text{ for $\mu$-a.e. }z \in \Omega.
        \end{align*} 
    In \cite{GLZ21-2}, the authors obtained the same results, but they could conclude that $\lim_{t \rightarrow 0^+}u(t, z) = h_0(z)$ for all $z \in \Omega$ because $d\mu = fdV$ for some strictly positive $f \in L^p(\Omega)$ for some $p > 1$. \\
    \indent We first remark that the Cauchy boundary condition on \cite{GLZ21-2} requires some relation between $\mu$ and $h_0$. Assume that the solution $u$ that we obtained from Theorem $4.5$ (or Theorem $4.3$) satisfies $\lim_{t \rightarrow 0^+}u(t, \cdot) = h_0$ in $L^1(\Omega, dV)$. While showing that $\lim_{t \rightarrow 0^+}u(t, \cdot) = h_0$ in $L^1(\Omega, d\mu)$, we proved that for any $\varepsilon > 0$, there exists a function $v_{\epsilon} \in \mathcal{P}(\Omega_T) \cap L^{\infty}(\Omega_T)$ such that $v_{\varepsilon} \leq u$ on $(0, t_0) \times \Omega$ for some sufficiently small $t_0$ and $v_{\varepsilon} \uparrow h_0-\varepsilon$ as $t \rightarrow 0^+$. Combining this with the main result in \cite{GT23}, 
    \begin{align*}
        u(t, \cdot) \rightarrow h_0 \text{ in capacity as } t \rightarrow 0^+.
    \end{align*} 
    This convergence implies that
    \begin{align*}
        (dd^cu(t, \cdot))^n = e^{\frac{\partial u}{\partial t}+F(t, z, u)}d\mu \rightarrow (dd^ch_0)^n \text{ weakly as } t \rightarrow 0^+.
    \end{align*}
    Thus for example, for any open subset $U \subset \Omega$, if either $\mu(U) = 0$ or $\lVert e^{\frac{\partial u}{\partial t}}\rVert_{L^1(U, d\mu)} \rightarrow 0$ as $t \rightarrow 0^+$, then $\int_U (dd^ch_0)^n = 0$. Conversely, for a compact subset $K \subset \Omega$ satisfying $\int_K (dd^ch_0)^n = 0$, we have either $\mu(K) = 0$ or $\lVert e^{\frac{\partial u}{\partial t}}\rVert_{L^1(K, d\mu)} \rightarrow 0$ as $t \rightarrow 0^+$. Therefore, we need an additional assumption on the pair $(\mu, h_0)$ to solve the Cauchy-Dirichlet problem with the pointwise Cauchy boundary condition. 
    \begin{definition}
        We say that the pair $(\mu, h_0)$ of a measure $\mu$ and a function $h_0$ is \textit{admissible} if for any $u \in PSH(\Omega) \in L^{\infty}(\Omega)$ with $u \leq h_0$ in $\partial \Omega$, $u \leq h_0$ $d\mu$-a.e. in $\Omega$ implies that $u \leq h_0$ in $\Omega$.
    \end{definition}
    If $(\mu, h_0)$ is admissible, the proof of Theorem $4.3$ implies that the solution $u$ that we obtained from Theorem $4.5$ (and Theorem $4.3$) satisfies $\lim_{t \rightarrow 0^+}u(t, z) = h_0(z)$ for all $z \in \Omega$, which is the original Cauchy boundary condition in \cite{GLZ21-2}. For example, if $d\mu = fdV$ for some $f \in L^1_{loc}(\Omega)$ satisfying $f > 0$ almost everywhere with respect to the Lebesgue measure on $\Omega$, then $(\mu, h_0)$ is admissible. Similarly, if $Cd\mu \geq (dd^ch_0)^n$ for some $C > 0$, then $(\mu, h_0)$ is admissible by \cite[Corollary 3.31]{GZ17}. We lastly remark that the admissible pairs in the sense of \cite[Definition 5.6]{EGZ15} are also admissible in the above sense.
\bibliographystyle{amsalpha}
\bibliography{ref.bib}
\end{document}